\documentclass{article}

\PassOptionsToPackage{round,compress}{natbib}
\usepackage[preprint]{neurips_2020}
\usepackage{enumitem}
\usepackage[algo2e,linesnumbered]{algorithm2e}

\usepackage[utf8]{inputenc} % allow utf-8 input
\usepackage[T1]{fontenc}    % use 8-bit T1 fonts
\usepackage{hyperref}       % hyperlinks
\usepackage{url}            % simple URL typesetting
\usepackage{booktabs}       % professional-quality tables
\usepackage{amsfonts}       % blackboard math symbols
\usepackage{nicefrac}       % compact symbols for 1/2, etc.
\usepackage{microtype}      % microtypography

\usepackage{amsmath,amsthm,amssymb,graphicx,dsfont,listings,color}
\usepackage{algorithm}
\usepackage{algpseudocode}
\usepackage{subeqnarray}
\usepackage{subcaption}
\usepackage{cleveref}
\hypersetup{
    colorlinks=true,
    citecolor= blue,
    linkcolor=blue
}

\allowdisplaybreaks

\newcommand{\E}{\mathbb{E}}
\newcommand{\PR}{\mathbb{P}}

\newcommand{\R}{\mathbb{R}}
\newcommand{\TV}{\mathrm{TV}}

\newcommand{\diag}{\textup{diag}}
\newcommand{\td}{\textup{TD}}

\newcommand{\subs}{\Xi^\theta}
\newcommand{\ssubs}{\tilde{\Xi}_i^\theta }
\newcommand{\rescale}{c' }

\newcommand{\khop}{{\color{black}\kappa}}
\newcommand{\rhok}{{\color{black}\rho^{\khop+1}}}

\newcommand{\nik}{{\color{black}N_i^{\khop}}}
\newcommand{\njk}{{\color{black}N_j^{\khop}}}
\newcommand{\nminusik}{{\color{black}N_{-i}^{\khop}}}

\newcommand{\sd}{\pi}
\newcommand{\SD}{D}

\newtheorem{assumption}{Assumption}
\newtheorem{lemma}{Lemma}
\newtheorem{theorem}{Theorem}
\newtheorem{corollary}{Corollary}
\newtheorem{definition}{Definition}

\newboolean{isfullversion}
\setboolean{isfullversion}{true}
\newcommand{\fvtest}[2]{\ifthenelse{\boolean{isfullversion}}{#1}{#2}}

\newboolean{showcomments}
\setboolean{showcomments}{false}
\newcommand{\lina}[1]{  \ifthenelse{\boolean{showcomments}}
	{ \textcolor{red}{(Lina says:  #1)}} {}  }
\newcommand{\guannan}[1]{  \ifthenelse{\boolean{showcomments}}
	{ \textcolor{blue}{(Guannan says:  #1)}} {}  }
\newcommand{\adam}[1]{  \ifthenelse{\boolean{showcomments}}
	{ \textcolor{red}{(Adam says:  #1)}} {}  }

\title{Scalable Multi-Agent Reinforcement Learning for Networked Systems with Average Reward}

\author{%
  Guannan Qu\\
  Caltech,
  Pasadena, CA \\
  \texttt{gqu@caltech.edu} \\
   \And
   Yiheng Lin \\
   Tsinghua University, Beijing, China \\
   \texttt{linyh16@mails.tsinghua.edu.cn}
      \And
   Adam Wierman \\
   Caltech, Pasadena, CA \\
   \texttt{adamw@caltech.edu}
      \And
   Na Li \\
   Harvard University, Cambridge, MA 02138 \\
   \texttt{nali@seas.harvard.edu}
}

\begin{document}

\maketitle

\begin{abstract}
It has long been recognized that multi-agent reinforcement learning (MARL) faces significant scalability issues due to the fact that the size of the state and action spaces are exponentially large in the number of agents. In this paper, we identify a rich class of networked MARL problems where the model exhibits a local dependence structure that allows it to be solved in a scalable manner. Specifically, we propose a Scalable Actor-Critic (SAC) method that can learn a near optimal localized policy for optimizing the average reward with complexity scaling with the state-action space size of local neighborhoods, as opposed to the entire network. Our result centers around identifying and exploiting an exponential decay property that ensures the effect of agents on each other decays exponentially fast in their graph distance. 
\end{abstract}

\section{Introduction}
%{\color{red} Potentially add Longbo as co-author}

As a result of its impressive performance in a wide array of domains such as game play \citep{silver2016mastering,mnih2015human}, robotics \citep{duan2016benchmarking}, and autonomous driving \citep{li2019reinforcement}, Reinforcement Learning (RL) has emerged as a promising tool for decision and control and there has been renewed interest in the use of RL in multi-agent systems, i.e., Multi-Agent RL (MARL). 

%computer vision \citep{caicedo2015active}, %\adam{I combined go and game play...we should have another category}, 
The multi-agent aspect of MARL creates additional challenges compared with single agent RL. One core challenge is scalability.  Even if individual agents' state or action spaces are small, the global state space or action space can take values from a set of size exponentially large in the number of agents. This ``curse of dimensionality'' renders the problem intractable in many cases. For example, RL algorithms such as temporal difference (TD) learning or $Q$-learning require storage of a $Q$-function \citep{bertsekas1996neuro} whose size is the same as the state-action space, which in MARL is exponentially large in $n$. Such scalability issues have been observed in the literature in a variety of settings, including \citet{complexity_blondel2000survey,complexity_papadimitriou1999complexity,zhang2019multi,kearns1999efficient,factor_guestrin2003efficient}. 
 
To address the issue of scalability, a promising approach that has emerged in recent years is to exploit problem structure, e.g., \citep{gu2020qlearning,qu2019exploiting,qu2019scalable}. One promising form of structure is enforcing local interactions, i.e., agents are associated with a graph and they interact only with nearby agents in the graph. Such local interactions are common in networked systems, including epidemics \citep{epi_mei2017dynamics}, social networks \citep{application_chakrabarti2008epidemic,application_llas2003nonequilibrium}, communication networks \citep{zocca2019temporal,application_communication}, queueing networks \citep{complexity_papadimitriou1999complexity}, smart transportation \citep{zhang2016control}, smart building systems \citep{application_wu2016optimal,zhang2017decentralized}. %Because of its wide applicability, local interaction have been considered broadly in various contexts in RL and beyond, e.g. in multi-agent combinatorial optimization \citep{gamarnik2013correlation,gamarnik2014correlation}, control system \citep{bamieh2002distributed,motee2008optimal}, and more recently, in a MARL setting \citet{qu2019scalable}. 
One powerful property associated with local interactions is the so-called \emph{exponential decay property} \citep{qu2019scalable}, also known as the \emph{correlation decay} \citep{gamarnik2013correlation,gamarnik2014correlation} or \emph{spatial decay} \citep{bamieh2002distributed,motee2008optimal}, which says that the impact of agents on each other decays exponentially in their graph distance. The exponential decay property often leads to potential for scalable, distributed algorithms for optimization and control \citep{gamarnik2013correlation,bamieh2002distributed,motee2008optimal}, and has proven effective for MARL, e.g.,  \citet{qu2019scalable}.

While exploiting local interactions and exponential decay in MARL has proven effective, results so far have been derived only in cases considering discounted total reward as the objective, e.g., \citep{qu2019scalable}. 
This is natural since results focusing on average reward, i.e., the reward in stationarity, are known to be more challenging to derive and require different techniques, even in the single agent RL setting \citep{tsitsiklis1999average}.
However, in many networked system applications, the average reward is a more natural objective. For example, in communication networks, the most common objective is the performance of the system (e.g. throughput) in stationarity. 

In this paper, our goal is to derive results for average reward MARL in networked systems.  However, it is unclear whether it is possible to obtain results that parallel those of \citet{qu2019scalable}, which focuses on discounted reward.  In the average reward case, the exponential decay property exploited by \citet{qu2019scalable} will no longer be true in general because the average reward case can capture certain NP-hard problems in the worst case. For example, \citet{complexity_blondel2000survey,complexity_whittle1988restless,complexity_papadimitriou1999complexity} all point out that such Markov Decision Process (MDP) with product state and action spaces is combinatorial in nature and is intractable from a computational complexity perspective in the worst case. 

\textbf{Contributions.} Despite the worst case intractability results, in this paper we show there are large classes of networked systems where average reward MARL is tractable.  More specifically, our main technical result (Theorem~\ref{thm:exp_decay}) shows that an exponential decay property still holds in the average reward setting under certain conditions that bound the interaction strength between the agents. These conditions are general and, numerically, exponential decay holds with high probability when the problem instance is generated randomly.  Given the presence of exponential decay, we develop a two time-scale actor-critic method \citep{rl_konda2000actor} and show that it finds a localized policy that is a $O(\rhok)$-approximation of a stationary point of the objective function (Theorem~\ref{thm:convergence}) with complexity that scales with the local state-action space size of the largest $\khop$-hop neighborhood, as opposed to the global state-action space size. 
To the best of our knowledge, this is the first such provable guarantee for scalable MARL with average reward for networked systems. 
Finally, to demonstrate the effectiveness of our approach we illustrate its performance in a setting motivated by protocol design for multi-access wireless communication. 

The key analytic technique underlying our results is a novel MDP perturbation result. We show that, when the local interaction between nodes is bounded, a perturbation on the state of an agent has diminishing affect on the state distribution of far away agents, which enables us to establish the exponential decay property that enables tractable learning. Our result has a similar flavor as some static problems in theoretic computer science, like the Lov\'{a}sz local lemma \citep{lovaszll}; or the ``correlation decay'' in \cite{gamarnik2014correlation} for solving combinatorial optimization problems. However, our result handles the more challenging dynamic setting where the states evolve over time. %{\color{red} it would be good to add a bit here if possible...say why ours is different/novel here } \guannan{done}

\textbf{Related Literature.} 
MARL dates back to the early work of \citet{marl_littman1994markov,marl_claus1998dynamics,marl_littman2001value,marl_nashq} (see \citet{marl_bu2008comprehensive} for a review) and has received considerable attention in recent years, e.g. \citet{zhang2018fully,kar2013cal,macua2015distributed,mathkar2017distributed,wai2018multi} and the review in \citet{zhang2019multi}. MARL considers widely-varying settings, including competitive agents and Markov games; however the setting most relevant to ours is cooperative MARL.  In cooperative MARL, the typical setting is that agents can choose their own actions but they share a common global state %and seek to maximize a global reward 
\citep{marl_bu2008comprehensive}. In contrast, we study a more structured setting where each agent has its own state that it acts upon. 
Despite the different setting, cooperative MARL problems still face scalability issues since the joint-action space is exponentially large. Methods have been proposed to deal with this, including independent learners 
%, where each agent employs single-agent RL methods, treating other agents as part of the environment 
\citep{marl_claus1998dynamics,matignon2012independent}, where each agent employs a single-agent RL method. While successful in some cases, such method can suffer from instability \citep{matignon2012independent}. 
Alternatively, one can approximate the large $Q$-table through linear function approximation \citep{zhang2018fully} or neural networks \citep{lowe2017multi}. Such methods can reduce computation complexity significantly, but it is unclear whether the performance loss caused by the function approximation is small. In contrast, our technique not only reduces computation but also guarantees small performance loss.

Our work adds to the growing literature that exploits exponential decay in the context of networked systems. The exponential decay concept we use is related to the concept of ``correlation decay'' studied in \citet{gamarnik2013correlation,gamarnik2014correlation}, though their focus is on solving static combinatorial optimization problems whereas ours is on learning polices in dynamic environments.  Most related to the current paper is \citet{qu2019scalable}, which considers the same networked MARL setting but focuses on discounted reward. 
The discount factor ensures that exponential decay always holds, in stark contrast to the more challenging average setting which is combinatorial in nature and is intractable in the worst case. 
{Additionally, our result in the average reward case improves the bound on the exponential decay rate in the discounted setting (Corollary~\ref{cor:discounted}). }

More broadly, this paper falls under the category of ``succinctly described'' MDPs in \citet[Section 5.2]{complexity_blondel2000survey}, which shows that when the state/action space is a product space formed by the individual state/action spaces of multiple agents, the resulting Markov Decision Process (MDP) is intractable in general, even when the problem has structure \citep{complexity_blondel2000survey,complexity_whittle1988restless,complexity_papadimitriou1999complexity}. In the context of these worst case complexity results, our work identifies conditions when it is possible to develop scalable methods. 

Finally, our work is related to \textit{factored MDPs} and \emph{weakly coupled MDPs}, though the model and the focus are very different. In factored MDPs, there is a global action affecting every agent whereas in our case, each agent has its own action. In weakly coupled MDPs, agents' transitions are decoupled \citep{weakly_mdp_meuleau1998solving} in stark contrast to our coupled setting.  

\section{Preliminaries}
In this section, we formally introduce the problem and define a key concept, the exponential decay property, that will be used throughout the rest of the paper.
\subsection{Problem Formulation}
We consider a network of $n$ agents that are associated with an underlying undirected graph $\mathcal{G} = (\mathcal{N},\mathcal{E})$, where  $\mathcal{N}=\{1,\ldots,n\}$ is the set of agents and $\mathcal{E}\subset \mathcal{N}\times\mathcal{N}$ is the set of edges.  Each agent $i$ is associated with state $s_i\in\mathcal{S}_i$, $a_i\in\mathcal{A}_i$ where $\mathcal{S}_i$ and $\mathcal{A}_i$ are finite sets. The global state is denoted as $s = (s_1,\ldots,s_n)\in \mathcal{S}:=\mathcal{S}_1\times\cdots\times \mathcal{S}_n$ and similarly the global action $a=(a_1,\ldots,a_n)\in\mathcal{A}:=\mathcal{A}_1\times\cdots\times\mathcal{A}_n$. At time $t$, given current state $s(t)$ and action $a(t)$, the next individual state $s_i(t+1)$ is independently generated and is only dependent on neighbors:
\begin{align}
  P(s(t+1)|s(t),a(t)) = \prod_{i=1}^n P_i(s_i(t+1)|s_{N_i}(t),a_i(t)),  \label{eq:transition_factor}
\end{align}
where notation $N_i$ means the neighborhood of $i$ (including $i$ itself) and $s_{N_i}$ is the states of $i$'s neighbors. In addition, for integer $\khop\geq 1$, we let $\nik$ denote the $\khop$-hop neighborhood of $i$, i.e. the nodes whose graph distance to $i$ is less than or equal to $\khop$, including $i$ itself. We also use $z = (s,a)$ and $\mathcal{Z} = \mathcal{S} \times\mathcal{A}$ to denote the state-action pair (space), and $z_i, \mathcal{Z}_i$ are defined analogously.
%We also let $\fk  = \sup_{i} |\nik|$. %\lina{$\khop$ starts with $0$ or $1$?}

%Next, a policy is a map from the state space $\mathcal{S}$ to the action space $\mathcal{A}$. 
Each agent is associated with a class of localized policies $\zeta_i^{\theta_i}$ parameterized by $\theta_i$. 
The localized policy $\zeta_i^{\theta_i}(a_i|s_i)$ is a distribution on the local action $a_i$ conditioned on the local state $s_i$, and each agent, conditioned on observing $s_i(t)$, takes an action $a_i(t)$ independently drawn from $\zeta_i^{\theta_i}(\cdot|s_i(t))$. %We will denote the local policy node $i$ as $\zeta_i$, and a local policy profile is the tuple of local policies at all nodes, 
We use $\theta = (\theta_1,\ldots,\theta_n)$ to denote the tuple of the localized policies $\zeta_i^{\theta_i}$, and use $\zeta^\theta(a|s) = \prod_{i=1}^n\zeta_i^{\theta_i}(a_i|s_i)$ to denote the joint policy. %, which is a product distribution of the localized policies as each agent acts independently. % $\zeta_i^{\theta_i}$ and write the $\zeta^\theta = (\zeta_1^{\theta_1},\ldots,\zeta_n^{\theta_n})$ where 

Each agent is also associated with a stage reward function $r_i(s_i,a_i)$ that depends on the local state and action, and the global stage reward is $r(s,a) = \frac{1}{n}\sum_{i=1}^n r_i(s_i,a_i)$. %We will also interpret $r_i$ as a vector. We assume that all rewards $r_i$ are bounded above by $\bar{r}$. 
The objective is to find localized policy tuple $\theta$ such that the global stage reward in stationarity is maximized, 
\begin{align}
     \max_\theta J(\theta): = \E_{(s,a) \sim \pi^\theta}   r(s,a), \label{eq:average_reward}
\end{align}
where $\pi^\theta$ is the distribution of the state-action pair in stationarity. We also define $J_i(\theta)$ to be
$J_i(\theta) =\E_{(s,a) \sim \pi^\theta}  r_i(s_i,a_i)$,
which satisfies $J(\theta) = \frac{1}{n}\sum_{i=1}^n J_i(\theta)$.

To provide context for the rest of the paper, we now review a few key concepts in RL. Fixing a localized policy $\theta$, the $Q$ function for the policy is defined as, 
{\small
\begin{align}
%V^\zeta(s) &= \E_{a_{i}(t) \sim \zeta_i(\cdot|s_i(t))}\bigg[  r(s(t), a(t) ) - g^\zeta  \bigg|s(0) = s \bigg]\\
%&= \sum_{i=1}^n  \E_{a_{i}(t) \sim \zeta_i(\cdot|s_i(t))} \bigg[ \sum_{t=0}^\infty \big( r_i(s_i(t) , a_i(t)) - g_i^\zeta  \big) \bigg| = s \bigg]\\
%&:=\sum_{i=1}^n V_i^\zeta(s) \\
Q^\theta(s,a) &= \E_{\theta}\bigg[ \sum_{t=0}^\infty \big( r(s(t), a(t) )  - J(\theta) \big) \bigg|s(0) = s, a(0) = a\bigg]\nonumber \\
&= \frac{1}{n}\sum_{i=1}^n  \E_{\theta} \bigg[ \sum_{t=0}^\infty \Big( r_i(s_i(t) ,a_i(t)) - J_i(\theta) \Big) \bigg|s(0) = s,a(0) = a \bigg] := \frac{1}{n} \sum_{i=1}^n Q_i^\theta(s,a) ,\label{eq:full_q}
\end{align}}where in the last step, we have also defined $Q_i^\theta$, the $Q$-function for the local reward. We note here that both $Q^\theta$ and $Q_i^\theta$ depends on the state-action pair of the whole network, and is thus intractable to compute and store.

Finally, we review the policy gradient theorem (Lemma~\ref{lem:policy_gradient}), an important tool for our analysis. It can be seen that the exact policy gradient relies on the full $Q$-function, which may be intractable to compute due to its large size.
\begin{lemma}[\citet{sutton2000policy}]\label{lem:policy_gradient}
Recall $\pi^\theta $ is the stationary distribution for the state-action pair under policy $\theta$.  
The gradient of $J(\theta)$ is then given by
\[\nabla_\theta J(\theta) =  \E_{(s,a)\sim \pi^\theta} Q^{\theta}(s,a) \nabla_\theta \log \zeta^\theta(a|s). \]
\end{lemma}

\subsection{Exponential Decay Property and Efficient Approximation of $Q$-Function}\label{sec:mainidea}
%\guannan{This section can potentially be merged to the previous section as part of the preliminaries}{\color{red} I like that idea!} \guannan{I am still debating the two options..} \lina{For your debate: Since Theorem 1 is at section 4, I think it makes sense to put this into preliminaries. You might call Section 3 as ``problem setup and preliminaries'' (or other title bigger than preliminaries). Section 3 should be devided ino two subsections one of which is the exponential decay property.}

Given the local interaction structure in the probability transition \eqref{eq:transition_factor}, a natural question to ask is that whether and under what circumstance, nodes that are far away have diminishing effect on each other. This has indeed been the subject of study in various contexts \citep{gamarnik2013correlation,gamarnik2014correlation,bamieh2002distributed,qu2019scalable}, where exponential decay properties or variants have been proposed and exploited for algorithmic design. In what follows, we provide the definition of the exponential decay property in the context of MARL and show its power in approximation of the exponentially large $Q$-function.

We define $\nminusik = \mathcal{N}/\nik$, i.e. the set of agents that are outside of node $i$'s $\khop$-hop neighborhood for some integer $\kappa$. 
We also write state $s$ as $(s_{\nik}, s_{N_{-i}^k})$, i.e. the states of agents that are in the $\khop$-hop neighborhood of $i$ and outside of the $\khop$-hop neighborhood respectively. 
The $(c,\rho)$ exponential decay property is defined below \citep{qu2019scalable}. 
\begin{definition}\label{def:exp_decay}
	The $(c,\rho)$ exponential decay property holds for $c>0, \rho\in(0,1)$, if for any $\theta$, for any $i\in\mathcal{N}$, $s_{\nik}\in \mathcal{S}_{\nik}$, $s_{\nminusik}, s_{\nminusik}'\in \mathcal{S}_{\nminusik} $, $a_{\nik} \in \mathcal{A}_{\nik}$, $a_{\nminusik}, a_{\nminusik}' \in \mathcal{A}_{\nminusik}$, the following holds,
	\begin{align}
	    \Big|Q_i^\theta(s_{\nik},s_{\nminusik}, a_{\nik}, a_{\nminusik}) - Q_i^\theta(s_{\nik},s_{\nminusik}', a_{\nik}, a_{\nminusik}')\Big| \leq c\rhok.  \label{eq:exp_decay_def}
	\end{align}
\end{definition}

The power of the exponential decay property is that it guarantees the dependence of $Q_i^\theta$ on other agents shrinks quickly as the distance between them grows. This naturally leads us to consider the following class of truncated $Q$-functions, where dependence on far-away nodes are ``truncated'',
{\small
\begin{align}
    \tilde{Q}_i^\theta (s_{\nik}, a_{\nik})  = \sum_{s_{\nminusik}\in \mathcal{S}_{\nminusik} , a_{\nminusik}\in \mathcal{A}_{\nminusik}} w_i(s_{\nminusik}, a_{\nminusik};s_{\nik}, a_{\nik} ) Q_i^\theta (s_{\nik},s_{\nminusik}, a_{\nik},a_{\nminusik}),\label{eq:truncated_q}
\end{align}}where $w_i(s_{\nminusik}, a_{\nminusik};s_{\nik}, a_{\nik} )$ are \textit{any} non-negative weights satisfying 
\begin{align}
    \sum_{s_{\nminusik}\in \mathcal{S}_{\nminusik} , a_{\nminusik}\in \mathcal{A}_{\nminusik}} w_i(s_{\nminusik}, a_{\nminusik};s_{\nik}, a_{\nik} ) = 1, \forall (s_{\nik}, a_{\nik}) \in \mathcal{S}_{N_{i}^k}\times \mathcal{A}_{N_{i}^k}. \label{eq:truncated_q_weights}
\end{align}
The following lemma shows that the exponential decay property guarantees the truncated $Q$-function \eqref{eq:truncated_q} is a good approximation of the full $Q$-function \eqref{eq:full_q}. Further, when using the truncated $Q$-functions in the policy gradient (Lemma~\ref{lem:policy_gradient}) in place of the full $Q$ function, the exponential decay property enables an accurate approximation of the policy gradient, which is otherwise intractable to compute in its original form in Lemma~\ref{lem:policy_gradient}. The proof of Lemma~\ref{lem:truncated_pg} can be found in \fvtest{\Cref{sec:truncated_pg}}{Appendix A.1} in the supplementary material.

\begin{lemma}\label{lem:truncated_pg} Under the $(c,\rho)$ exponential decay property, for any truncated $Q$-function in the form of \eqref{eq:truncated_q}, the following holds. 
\begin{itemize}
    \item[(a)] $\forall (s,a)\in \mathcal{S}\times\mathcal{A}$, $|Q_i^\theta(s,a) - \tilde{Q}_i^\theta(s_{\nik},a_{\nik})| \leq c\rhok$.
    \item[(b)] Define the following approximated policy gradient,
\begin{align}
    \hat{h}_i(\theta) = \E_{(s,a) \sim\pi^\theta} \Big[\frac{1}{n}\sum_{j\in {\nik}} \tilde{Q}_j^{\theta} (s_{\njk},a_{\njk})\Big]\nabla_{\theta_i} \log\zeta_i^{\theta_i} (a_i|s_i) . \label{eq:truncated_pg}
\end{align}
 Then, if $\Vert \nabla_{\theta_i} \log \zeta_i^{\theta_i} (a_i|s_i)\Vert\leq L_i$ for any $a_i$, $s_i$, we have
$\Vert \hat{h}_i(\theta) - \nabla_{\theta_i} J(\theta)\Vert \leq  cL_i \rhok $.

\end{itemize}
\end{lemma}

The above results show the power of the exponential decay property -- the $Q$ function and the policy gradient can be efficiently approximated despite their exponentially large dimension. These properties can be exploited to design scalable RL algorithms for networked systems.%, as shown by \citet{qu2019scalable}. 

While the exponential decay property is powerful, we have yet to show it actually holds under the average reward setting of this paper. In \citet{qu2019scalable}, it was shown that the exponential decay property always holds when the objective is the discounted total reward. It turns out the average reward case is fundamentally different. In fact, it is not hard to encode NP-hard problems like graph coloring, 3-SAT into our setup (for an example, see \fvtest{\Cref{sec:nphard}}{Section A.2} in the supplementary material). As such, our problem is intractable in general, and there is no hope for the exponential decay property to hold generally. This is perhaps not surprising given the literature on MDPs with product state/action spaces \citep{complexity_blondel2000survey,complexity_papadimitriou1999complexity}, which all point out the combinatorial nature and intractability of such multi-agent MDP problems. 

Despite the worst case intractability, in the next section we show that a large class of subproblems in fact do satisfy the exponential decay property, for which we design a scalable RL algorithm utilizing the approximation guarantees in Lemma~\ref{lem:truncated_pg} implied by the exponential decay property. 

%\textbf{Example 1: graph coloring.} There exists a polynomial-time reduction from the problem to the graph coloring problem.  
%\lina{so I think you have already have the proof? if it is the case, you might just present the reduction to show NP-hardness? But my question/confusion, it is kind of ``obvious'' that the problem is hard. do we need the statement and the proof?} \guannan{This is also what I feel - the difficulty is kinda obvious, so it is not worth the effort to give a formal complexity proof. Just wanna present the connection in a informal way...  }

\section{Main Result}
Despite the worst case intractability, in this section we identify a large subclass where learning is tractable, and illustrate this by developing a scalable RL algorithm for such problems. To this end, in \Cref{subsec:exponential_decay}, we identify a general condition on the interaction among the agents under which the exponential decay property holds. Then, the exponential decay property, together with the idea of truncated $Q$-function and policy gradient outlined in \Cref{sec:mainidea}, is combined with the actor critic framework in \Cref{subsec:algo} to develop a Scalable Actor Critic algorithm that can find a $O(\rhok)$-approximate stationary point of the objective function $J(\theta)$ with complexity scaling with the state-action space size of the largest $\kappa$-hop neighborhood.

\subsection{Exponential Decay Holds when Interactions are Bounded}\label{subsec:exponential_decay}

Our first result is our most technical and it identifies a rich class of problems when the exponential decay property holds.
%The proof of Theorem~\ref{thm:exp_decay} can be found in \Cref{sec:proof_exp_decay} in the supplementary material.

\begin{theorem}\label{thm:exp_decay}
Define 
\[ {C}_{ij} = \left\{ \begin{array}{ll}
0, & \text{ if }j\notin N_i,\\
\sup_{s_{N_i/j},a_i}\sup_{s_j,s_j'} \TV(  P_i(\cdot|s_j, s_{N_i/j},a_i) , P_i(\cdot| s_j', s_{N_i/j},a_i)   ),  & \text{ if } j\in N_i, j\neq i,\\
\sup_{s_{N_i/i}}\sup_{s_i,s_i',a_i,a_i'} \TV(  P_i(\cdot|s_i, s_{N_i/i},a_i) , P_i(\cdot| s_i', s_{N_i/i},a_i')   ), & \text{ if }j=i,
\end{array} \right.  \]
where $\TV(\cdot,\cdot)$ is the total variation distance bewteen two distributions. 
If for all $i\in \mathcal{N}$, $\sum_{j=1}^n C_{ij} \leq \rho <1$ and $| r_i(s_i,a_i)|\leq \bar{r}, \forall (s_i,a_i)\in \mathcal{S}_i\times\mathcal{A}_i$, then the $(\frac{\bar{r}}{1-\rho},\rho)$ exponential decay property holds. 
\end{theorem}
%\guannan{can replace the above by a bound on spectral radius; not sure if it helps}

The proof of Theorem~\ref{thm:exp_decay} can be found in \fvtest{\Cref{sec:proof_exp_decay}}{Appendix B.1} in the supplementary material. The key analytic technique underlying our the proof is a novel MDP perturbation result (\fvtest{Lemma~\ref{lem:markov_chain_exp_decay} in \Cref{sec:proof_exp_decay}}{Lemma~3 in Appendix B.1} in the supplementary material). We show that, under the condition in Theorem~\ref{thm:exp_decay}, a perturbation on the state of an agent has exponential decaying affect on the state distribution of far away agents, which enables us to establish the exponential decay property.

In Theorem~\ref{thm:exp_decay}, $C_{ij}$ is the maximum possible change of the distribution of node $i$'next state as a result of a change in node $j$'th current state, where ``change'' is measured in total-variation distance. In other words, $C_{ij}$ can be interpreted as the strength of node $j$'s interaction with node $i$. With this interpretation, Theorem~\ref{thm:exp_decay} shows that, for the exponential decay property to hold, we need (a) the interaction strength between each pair of nodes to be small enough; and (b) each node to have a small enough neighborhood. This is consistent with related conditions in the literature in exponential decay in static combinatorial optimizations, e.g. in \cite{gamarnik2013correlation,gamarnik2014correlation}, which requires the product between the maximum ``interaction'' and the maximum degree is bounded. 

\textit{Relation to discounted reward case. } In the $\gamma$-discounted total reward case in \citet{qu2019scalable}, the exponential decay property is automatically true with decay rate $\gamma$ because for an agent to affect another far-way agent, it takes many time steps for the effect to take place, at which time the effect will have been significantly dampened as a result of the discounting factor. 
This stands in contrast with the average reward case, where it is unclear whether the exponential decay property holds or not as there is no discounting anymore. 
As a result, the analysis for the average reward turns out to be more difficult and requires very different techniques. 
Having said that, the result in the average reward case also has implications for the discounted reward case. Specifically, the results in Theorem~\ref{thm:exp_decay} also leads to a sharper bound on the decay rate in the exponential decay property for the discounted reward case, showing the decay rate can be strictly smaller than the discounting factor $\gamma$. This is stated in the following corollary, and we provide a more detailed explanation on its setting and proof in \fvtest{\Cref{subsec:discounted}}{Appendix~B.3} in the supplementary material. 
\begin{corollary}\label{cor:discounted}
Under the conditions in Theorem~\ref{thm:exp_decay}, the $(\frac{\bar{r}}{1 - \gamma\rho},\rho\gamma) $ exponential decay property holds for the $\gamma$-discounted reward case. 
\end{corollary}

%\adam{Add a paragraph here highlighting contrast with the discounted case. Discount fractor makes exponential decay hold.  In average case it is much less clear and analysis is more difficult...  To emphasize this notice that Thm 1 provides improved bound in the discounted case. }  \guannan{done}
%\guannan{Referecing discounted, highlight differences; discounted kinda obvious, average is much more difficult... } \guannan{say this can improve discounted reward exponential decay as well... . }

Finally, beyond the theoretic result, we also numerically verify that the exponential decay property holds widely for many randomly generated problem instances. We report these findings in \fvtest{\Cref{sec:example_exp_decay}}{Appendix B.4} in the supplementary material.

\subsection{Scalable Actor-Critic Algorithm and Convergence Guarantee}\label{subsec:algo}
Given the establishment of the exponential decay property in the previous section, a natural idea for algorithm design is to first estimate a truncated $Q$ function \eqref{eq:truncated_q}, compute the approximated policy gradient \eqref{eq:truncated_pg} and then do gradient step. In the following, we combine these ideas with the actor-critic framework in \citet{rl_konda2000actor} which uses Temporal Difference (TD) learning to estimate the truncated $Q$-function, and present a Scalable Actor Critic Algorithm in Algorithm~\ref{algorithm:sac}. In Algorithm~\ref{algorithm:sac}, each step is conducted by all agents but is described from agent $i$'s perspective, and
for the ease of presentation, we have used $z$ ($\mathcal{Z}$) to represent state-action pairs (spaces), e.g. $z_i(t) = (s_i(t),a_i(t)) \in \mathcal{Z}_i = \mathcal{S}_i\times\mathcal{A}_i$, $z_{\nik}(t) = (s_{\nik}(t), a_{\nik}(t))\in \mathcal{Z}_{\nik} = \mathcal{S}_{\nik}\times \mathcal{A}_{\nik}$. 
The algorithm runs on a single trajectory and, at each iteration, it consists of two main steps, the critic (step~\ref{algo:critic_1} and \ref{algo:critic_2}) and the actor (step~\ref{algo:actor_1} and \ref{algo:actor_2}), which we describe in detail below.

\begin{itemize}[leftmargin=18pt]

\item \textbf{The Critic. } In line with TD learning for average reward \citep{tsitsiklis1999average}, the critic consists of two variables, $\hat{\mu}_i^t$, whose purpose is to estimate the average reward; and $\hat{Q}_i^t$, the truncated $Q$ function, which follows the standard TD update. We note here that $\hat{Q}_i^t$ is one dimension less than the $\tilde{Q}_i$ defined in \eqref{eq:truncated_q}. In more detail, $\tilde{Q}_i \in \R^{\mathcal{Z}_\nik}$ is defined on all state-action pairs in the neighborhood, but for $\hat{Q}_i^t$ we select an arbitrary dummy state-action pair $\tilde{z}_{\nik} \in \mathcal{Z}_\nik $, and $\hat{Q}_i^t \in \R^{\hat{\mathcal{Z}}_{\nik}} $ is only defined on $\hat{\mathcal{Z}}_{\nik} = \mathcal{Z}_{\nik} / \{\tilde{z}_\nik\}$. In other words, $\hat{Q}_i^t(z_{\nik})$ is not defined for $z_{\nik} = \tilde{z}_{\nik}$, and when encountered,  $\hat{Q}_i^t(\tilde{z}_{\nik})$ is considered as $0$. The reason for introducing such dummy state-action pair is mainly technical and is standard in value iteration for the average reward MDP, see e.g. \cite{bertsekas2007dpbook}. 

\item \textbf{The Actor. } The actor computes the approximated gradient using the truncated $Q$-function according to Lemma~\ref{eq:truncated_pg} and follows a gradient step. Note that the step size in the gradient step contains a rescaling scalar factor $\Gamma(\mathbf{\hat{Q}}^t)$ which depends on the truncated $Q$-functions $\mathbf{\hat{Q}}^t = \{ \hat{Q}_i^t\}_{i=1}^n$. Such a rescaling factor is used in \citet{rl_konda2000actor} and is mainly an artifact of the proof used to guarantee the actor does not move too fast when $\hat{Q}_i^t$ is large. In numerical experiments, we do not use this rescaling factor. 

\end{itemize}

\begin{algorithm}\caption{Scalable Actor-Critic (SAC)}\label{algorithm:sac}
\DontPrintSemicolon
	    \KwIn{$\theta_i(0)$; parameter $\kappa$; step size sequence $\{\alpha_t,\eta_t\}$.}
	    Initialize $\hat{\mu}_i^0 = 0$ and $\hat{Q}_i^0$ to be the all zero vector in $\R^{\hat{\mathcal{Z}}_{\nik}}$; start from random $z_i(0) =(s_i(0),a_i(0))$ \;
	    
	    \For{$t=0, 1,2,\ldots$}{
	        Receive reward $r_i(t) = r_i(s_i(t),a_i(t))$.\;
	        
	         Get state $s_i(t+1)$, take action $a_i(t+1) \sim \zeta_i^{\theta_i(t)}(\cdot|s_i(t+1))$\; %and denote $z_i(t+1) = (s_i(t+1),a_i(t+1))$.\; 

	         \tcc{The critic.}
	         Update average reward 
	        	$\hat{\mu}_i^{t+1} = (1-\alpha_t) \hat{\mu}_i^{t} +\alpha_t r_i(t)$.\; \label{algo:critic_1}
	        	
	         Update truncated $Q$ function. If $z_{\nik}(t) = \tilde{z}_{\nik}$, then set $\hat{Q}_i^{t+1}(z_{\nik}) = \hat{Q}_i^t(z_{\nik}), \forall z_{\nik}\in \hat{\mathcal{Z}}_{\nik}$.  Otherwise, set 
	        	        \begin{align*}
	        \hat{Q}_i^{t+1}(z_{\nik}(t)) &= (1 - \alpha_{t}) \hat{Q}_i^{t}(z_{\nik}(t))  + \alpha_{t} (r_i(t) - \hat{\mu}_i^t +  \hat{Q}_i^{t}(z_{\nik}(t+1))), \\
	        \hat{Q}_i^{t+1}(z_{\nik}) &=  \hat{Q}_i^{t}(z_{\nik} ), \text{\quad  for   } z_{\nik}\in \hat{\mathcal{Z}}_\nik /\{ z_{\nik}(t)\},	            
	        \end{align*} 
    with the understanding that $\hat{Q}_i^t(\tilde{z}_{\nik}) = 0$.  \label{algo:critic_2}
	       % \begin{align*}
	       % \hat{Q}_i^{t+1}(s_{\nik}(t),a_{\nik}(t)) &= (1 - \alpha_{t}) \hat{Q}_i^{t}(s_{\nik}(t),a_{\nik}(t)) \\
	       % &\quad + \alpha_{t} (r_i(t) - \mu_i^t +  \hat{Q}_i^{t}(s_{\nik}(t+1),a_{\nik}(t+1)) ) \\
	       % \hat{Q}_i^{t+1}(s_{\nik},a_{\nik}) &=  \hat{Q}_i^{t}(s_{\nik} ,a_{\nik} ) \text{ for } (s_{\nik},a_{\nik})\neq (s_{\nik},a_{\nik}) 	            
	       % \end{align*}
	      	      
	      	         \tcc{The actor.}
	         Calculate approximated gradient
	        $\hat{g}_i(t) =  \nabla_{\theta_i} \log \zeta_i^{\theta_i(t)}(a_i(t)|s_i(t)) \frac{1}{n} \sum_{j \in \nik}  \hat Q_j^t(z_{\njk}(t))  $.\; \label{algo:actor_1}
	        
	         Conduct gradient step $\theta_i(t+1) = \theta_i(t) + \beta_t \hat{g}_i(t)$, where $\beta_t = \eta_t \Gamma(\mathbf{\hat{Q}}^t)$ and $\Gamma(\mathbf{\hat{Q}}^t) = \frac{1}{1 + \max_{j} \Vert\hat{Q}_j^t \Vert_\infty}$ is a rescaling scalar.  \label{algo:actor_2}
	    }
\end{algorithm}

In the following, we prove a convergence guarantee for Scalable Acotr Critic method introduced above. To that end, we first describe the assumptions in our result. Our first assumption is that the rewards are bounded, a standard assumption in RL \citep{tsitsiklis1997analysis}. 
\begin{assumption}\label{assump:reward} For all $i$, and $s_i\in \mathcal{S}_i, a_i\in \mathcal{A}_i$, we have $0\leq r_i(s_i,a_i)\leq \bar{r}$.
\end{assumption}
Our next assumption is the exponential decay property, which as shown in \Cref{subsec:exponential_decay}, holds broadly for a large class of problems. 
\begin{assumption} \label{assump:exp_decay} The $(c,\rho)$ exponential decay property holds.
\end{assumption}
Our next assumption is on the step size and is standard \citep{rl_konda2000actor}. We note that our algorithm uses two-time scales, meaning the actor progresses slower than the critic. 

\begin{assumption} \label{assump:stepsize}The positive step sizes $\alpha_t,\eta_t$ are deterministic, non-increasing, square summable but not summable\footnote{A sequence $\alpha_t$ is square summable if $\sum_{t=0}^\infty \alpha_t^2<\infty$; is \emph{not} summable if $\sum_{t=0}^\infty \alpha_t = \infty$. } and satisfy $\sum_{t=0}^\infty (\frac{\eta_t}{\alpha_t})^d <\infty$ for some $d>0$.  
\end{assumption}

Our next assumption is that the underlying problem is uniformly ergodic, which is again standard \citep{rl_konda2000actor}.

\begin{assumption} \label{assump:ergodicity}
Under any policy $\theta$, the induced Markov chain over the state-action space $\mathcal{Z} =\mathcal{S}\times\mathcal{A}$ is ergodic with stationary distribution $\sd^\theta$. Further, (a) For all $z \in\mathcal{Z}$, $\pi^\theta(z)\geq \sigma $ for some $\sigma>0$. (b) $\Vert P^\theta - \mathbf{1} (\sd^\theta)^\top\Vert_{D^\theta} \leq \mu_D$ for some $\mu_D\in (0,1)$, where $\SD^\theta = \diag(\sd^\theta) \in\R^{\mathcal{Z}\times\mathcal{Z}}$ and $\Vert\cdot\Vert_{\SD^\theta}$ is the weighted Euclidean norm $\Vert x \Vert_{\SD^\theta} = \sqrt{x^\top \SD^\theta x} $ for vectors $x\in\R^{\mathcal{Z}}$, and the corresponding induced norm for matrices.
\end{assumption}
Note that, for Markov chains on the state-action pair that are ergodic, Assumption~\ref{assump:ergodicity} are automatically true for some $\sigma>0$ and $\mu_D \in (0,1)$.\footnote{Assumption~\ref{assump:ergodicity}(b) is standard in the study of average reward TD learning, e.g. \citet[Sec. 4.3]{tsitsiklis1999average}.} Assumption~\ref{assump:ergodicity} also requires that it is true with constant $\sigma, \mu_D$ holds uniformly for all $\theta$. Such a uniform ergodicity assumption is common in the literature on actor-critic methods, e.g. \citet{konda2003linear}. Lastly, we assume the gradient is bounded and is Lipschitz continuous, again a standard assumption \citep{qu2019scalable}.

\begin{assumption} \label{assump:gradient} For each $i$, $a_i,s_i$, $\Vert \nabla_{\theta_i} \log \zeta_i^{\theta_i}(a_i|s_i)\Vert\leq L_i$. We let $L := \sqrt{\sum_{i\in\mathcal{N}} L_i^2}$. Further, $ \nabla_{\theta_i} \log \zeta_i^{\theta_i}(a_i|s_i)$ is $L_i'$-Lipschitz continuous in $\theta_i$, and $\nabla J(\theta)$ is $L'$-Lipschitz continuous in $\theta$. %{\color{red} where $L'=...$}. 
\end{assumption}
%\lina{$L'$ is not used in Theorem 2. I guess it is used in the proof and the value of $L'$ won't affect the theorem, but affecting the convergence rate. After the theorem, it would be good to explain the role of $L$ and and $L'$. Note that $L$ is represented in the upper bound. }
We can now state our main convergence result. Note that the guarantee is asymptotic in nature.  This is to be expected since, to the best of our knowledge, the finite time performance of two-time scale actor-critic algorithms on a single trajectory has long remained an open problem until the very recent progress of \citep{wu2020finite,xu2020non}. We leave a finite time analysis of our algorithm as future work. 

\begin{theorem}\label{thm:convergence}
Under Assumptions~\ref{assump:reward}-\ref{assump:gradient}, we have 
$
    \liminf_{T\rightarrow\infty} \Vert \nabla J(\theta(T))\Vert \leq  L \frac{  c \rhok}{1-\mu_{D}}.
$
\end{theorem}

Briefly speaking, Theorem~\ref{thm:convergence} is a consequence of the power of exponential decay property in approximating the $Q$-function (Lemma~\ref{lem:truncated_pg}), and its proof also uses tools from the stochastic approximation literature \citep{konda2003linear}. The complete proof of Theorem~\ref{thm:convergence} is provided in 
\fvtest{\Cref{sec:actor}}{Appendix D} in the supplementary material. 
We comment that the asymptotic bound in Theorem~\ref{thm:convergence} does not depend on parameters like $L',\bar{r}$, but these parameters will affect the convergence rate of the algorithm. 
From Theorem~\ref{thm:convergence}, our Scalable Actor Critic algorithm can find an approximated stationary point with gradient size $O(\rhok)$ that decays exponentially in $\kappa$. 
Therefore, even for a small $\kappa$, our algorithm can find an approximate local minimizer in a \emph{scalable} manner, with the complexity of the algorithm only scales with the state-action space size of the largest $\khop$-hop neighborhood due to the use of the truncated $Q$ function, which could be much smaller than the full state-action space size (which are exponentially large in $n$) when the graph is sparse. 
%\lina{This paragraph doesn't read as proof paragraph. Following my previous comment, you can try to explain the role of the parameters in Assumption 1-5, which help readers gain insight of the proof. }

\section{Numerical Experiments}
We consider a wireless network with multiple access points \citep{zocca2019temporal}, where there is a set of users  $U = \{u_1, u_2, \cdots, u_n\},$ and a set of the network access points $Y = \{y_1, y_2, \cdots, y_m\}$. Each user $u_i$ only has access to a subset $Y_i \subseteq Y$ of the access points. % which is determined by her/his physical location and the facilities around this location. 
At each time step, with probability $p_i$, user $u_i$ receives a new packet with deadline $d_i$, which is the number of steps when the packet will be discarded if not transmitted yet. 
%.\footnote{Here the deadline means the packet expires and is discarded if it is not transmitted in $d_i$ time steps.} 
Then user $u_i$ can choose to send one of the packets in its queue to one of the access points $y_k$ in its available set $ Y_i$. 
If no other users select this same access point, then the packet is transmitted with success probability $q_k$ depending on the access point; however, if two users choose to send packets to the same access point, neither packet is sent. A user receives a local reward of $1$ once successfully sending a packet. 
Such a setting fits to our framework, as the users only interact locally, where ``local'' is defined by the conflict graph, in which two users $u_i$ and $u_j$ are neighbors if and only if they share an access point, i.e. $Y_i\cap Y_j\neq \emptyset$. 
We leave a detailed description of the state/action space, transition probability and reward to \fvtest{\Cref{sec:numerical_detail}}{Appendix E} in the supplementary material.

In this setting, we consider a grid of $5\times 5$ users  in Figure~\ref{fig:communication_grid}, where each user has access points on the corners of its area. We run the Scalable Actor Critic algorithm with $\khop=1$ to learn a localized soft-max policy. 
We compare the proposed method with a benchmark based on the localized ALOHA protocol \citep{aloha}, where each user has a certain probability of sending the earliest packet and otherwise not sending at all. 
When it sends, it sends the packet to a random access point in its available set, with probability proportion to the success probability of this access point and inverse proportion to the number of users that share this access point.
The results are shown in Figure~\ref{fig:comm_ave_reward_5by5}. 
Due to space constraints, we leave the details of the experiment to \fvtest{\Cref{sec:numerical_detail}}{Appendix E} in the supplementary material. 
%\lina{since one main result is the exponential decaying property. Note that theorem 2 also talks about exponential property. I do think we should have a figure demonstrating this result, even if this simulation is done using a toy example and a line.  }

\begin{figure}
  \begin{minipage}[b]{0.5\textwidth}
    \centering
    \includegraphics[width = .8\textwidth]{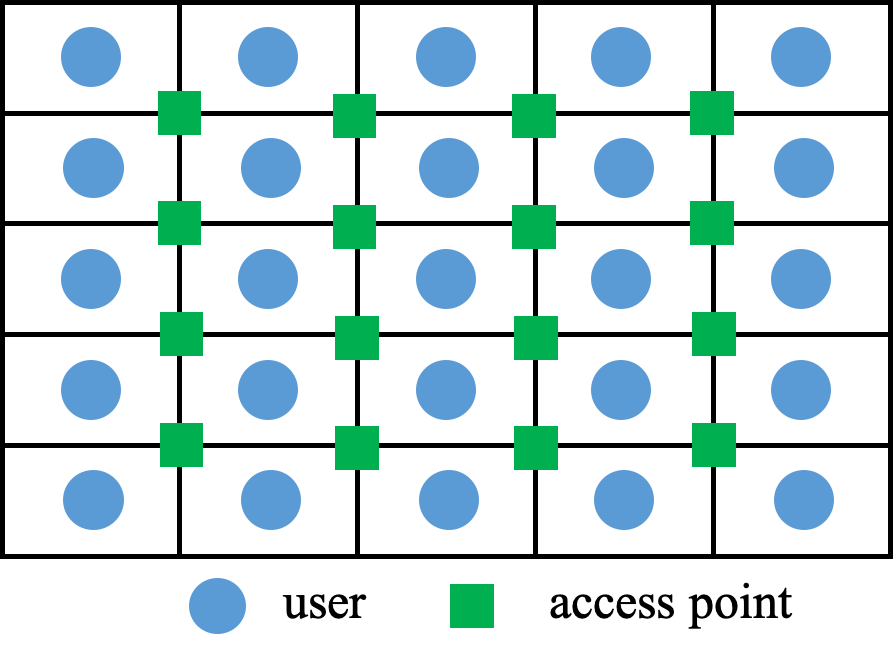}
    \caption{Setup of users and access points.}
    \label{fig:communication_grid}
  \end{minipage}~
  \begin{minipage}[b]{0.5\textwidth}
    \centering
    \includegraphics[width = \textwidth]{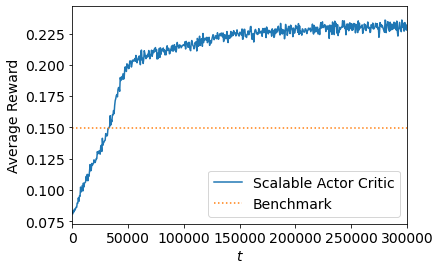}
    \caption{Average reward over the training process. }
    \label{fig:comm_ave_reward_5by5}
  \end{minipage}
\end{figure}

\newpage
\section*{Broader Impact}
This paper contributes to the theoretical foundations of multi-agent reinforcement learning, with the goal of developing tools that can apply to the control of networked systems. The work can potentially lead to RL-based algorithms for the adaptive control of cyber-physical systems, such as the power grid, smart traffic systems, communication systems, and other smart infrastructure systems.  While the approach is promising, as with other all theoretical work, it is limited by its assumptions.  Applications of the proposed algorithm in its current form should be considered cautiously since the analysis here focuses on efficiency and does not consider the issue of fairness.  

We see no ethical concerns related to this paper.  

\bibliographystyle{plainnat}
\bibliography{references}

\newpage

\fvtest{
\section*{Supplementary Materials}
\appendix

\section{Appendix to \Cref{sec:mainidea}}
\subsection{Proof of Lemma~\ref{lem:truncated_pg}} \label{sec:truncated_pg}
We first show part (a) that the truncated $Q$ function is a good approximation of the true $Q$ function. To see that, we have for any $(s,a)\in\mathcal{S}\times \mathcal{A}$, by \eqref{eq:truncated_q} and \eqref{eq:truncated_q_weights}, 
\begin{align}
&|\tilde{Q}_i^\theta (s_{\nik}, a_{\nik}) - Q_i^\theta(s,a)| \nonumber  \\
&= \Big| \sum_{s_{\nminusik}', a_{\nminusik}'} w_i(s_{\nminusik}', a_{\nminusik}';s_{\nik}, a_{\nik} ) Q_i^\theta (s_{\nik},s_{\nminusik}', a_{\nik},a_{\nminusik}') - Q_i^\theta(s_{\nik},s_{\nminusik},a_{\nik},a_{\nminusik}) \Big|\nonumber \\
&\leq \sum_{s_{\nminusik}', a_{\nminusik}'} w_i(s_{\nminusik}', a_{\nminusik}';s_{\nik}, a_{\nik} ) \Big| Q_i^\theta (s_{\nik},s_{\nminusik}', a_{\nik},a_{\nminusik}') - Q_i^\theta(s_{\nik},s_{\nminusik},a_{\nik},a_{\nminusik}) \Big|\nonumber\\
&\leq c\rhok,\label{appendix:truncated:eq:q_err}
\end{align}
where in the last step, we have used the $(c,\rho)$ exponential decay property, cf. Definition~\ref{def:exp_decay}. 

Next, we show part (b). Recall by the policy gradient theorem (Lemma~\ref{lem:policy_gradient}),
\begin{align*}
    \nabla_{\theta_i} J(\theta) &=  \E_{(s,a)\sim \pi^\theta} Q^{\theta}(s,a) \nabla_{\theta_i} \log \zeta^\theta (a|s)
    = \E_{(s,a)\sim\pi^\theta} Q^{\theta}(s,a) \nabla_{\theta_i} \log \zeta_i^{\theta_i} (a_i|s_i),
\end{align*}
where we have used $\nabla_{\theta_i}\log\zeta^{\theta}(a|s) =\nabla_{\theta_i} \sum_{j\in\mathcal{N}} \log\zeta_j^{\theta_j}(a_j|s_j)=\nabla_{\theta_i}\log\zeta_i^{\theta_i}(a_i|s_i)$ by the localized policy structure. With the above equation, we can compute $\hat{h}_i(\theta) - \nabla_{\theta_i} J(\theta)$,
\begin{align*}
    \hat{h}_i(\theta) - \nabla_{\theta_i} J(\theta) 
    & =  \E_{(s,a)\sim\pi^\theta} \Big[\frac{1}{n}\sum_{j\in {\nik}} \tilde{Q}_j^{\theta} (s_{\njk},a_{\njk}) -  Q^{\theta}(s,a)\Big] \nabla_{\theta_i}  \log \zeta_i^{\theta_i}(a_i|s_i)\\
    &=  \E_{(s,a)\sim\pi^\theta} \Big[\frac{1}{n}\sum_{j\in \mathcal{N}} \tilde{Q}_j^{\theta} (s_{\njk},a_{\njk}) - \frac{1}{n}\sum_{j\in \mathcal{N}} Q^{\theta}_j(s,a)\Big] \nabla_{\theta_i}  \log \zeta_i^{\theta_i} (a_i|s_i) \\
    &\qquad - \E_{(s,a)\sim\pi^\theta} \frac{1}{n}\sum_{j\in \nminusik} \tilde{Q}_j^{\theta} (s_{\njk},a_{\njk}) \nabla_{\theta_i}  \log \zeta_i^{\theta_i} (a_i|s_i)\\
    &:= E_1 - E_2.
\end{align*}
We claim that $E_2 = 0$. To show this, note that $\pi^\theta(s,a) = d^\theta(s) \prod_{\ell=1}^n \zeta^{\theta_\ell}_\ell(a_\ell|s_\ell) $, where $d^\theta$ is the sationary distribution of the state. Then, for any $j\in \nminusik$, we have,
    \begin{align}
        &\mathbb{E}_{(s,a)\sim\pi^\theta}  \nabla_{\theta_i}  \log \zeta_i^{\theta_i}(a_i|s_i)    \tilde{Q}_j^{\theta}(s_{\njk},a_{\njk}) \nonumber \\
        &= \sum_{s,a} d^\theta(s) \prod_{\ell=1}^n \zeta^{\theta_\ell}_\ell(a_\ell|s_\ell) \frac{\nabla_{\theta_i}  \zeta_i^{\theta_i}(a_i|s_i)}{\zeta_i^{\theta_i}(a_i|s_i)} \tilde{Q}_j^{\theta}(s_{\njk},a_{\njk})\nonumber\\
        &= \sum_{s,a} d^\theta(s) \prod_{\ell\neq i} \zeta^{\theta_\ell}_\ell(a_\ell|s_\ell) \nabla_{\theta_i}  \zeta_i^{\theta_i}(a_i|s_i)  \tilde{Q}_j^{\theta}(s_{\njk},a_{\njk})\nonumber\\
        &= \sum_{s,a_{1},\ldots,a_{i-1},a_{i+1},\ldots,a_n } d^\theta(s) \prod_{\ell\neq i} \zeta^{\theta_\ell}_\ell(a_\ell|s_\ell)  \tilde{Q}_j^{\theta}(s_{\njk},a_{\njk}) \sum_{a_i} \nabla_{\theta_i}   \zeta_i^{\theta_i}(a_i|s_i) \nonumber\\
        & = 0,\label{eq:truncated_pg_bias}
    \end{align}
    where in the last equality, we have used $\tilde{Q}_j^{\theta}(s_{\njk},a_{\njk})$ does not depend on $a_i$ as $i\not\in \njk$; and $\sum_{a_i} \nabla_{\theta_i}   \zeta_i^{\theta_i}(a_i|s_i) =  \nabla_{\theta_i}   \sum_{a_i} \zeta_i^{\theta_i}(a_i|s_i) =\nabla_{\theta_i} 1 = 0 $. Now that we have shown $E_2=0$, we can bound $E_1$ as follows,
    \begin{align*}
        \Vert\hat{h}_i(\theta) - \nabla_{\theta_i} J(\theta)\Vert = \Vert E_1\Vert & \leq  \E_{(s,a)\sim\pi^\theta} \frac{1}{n} \sum_{j\in \mathcal{N}} \Big|\tilde{Q}_j^{\theta} (s_{\njk},a_{\njk}) - Q^{\theta}_j(s,a)\Big| \Vert \nabla_{\theta_i}  \log \zeta_i^{\theta_i} (a_i|s_i)\Vert\\
       &\leq  c\rhok L_i,
    \end{align*}
where in the last step, we have used \eqref{appendix:truncated:eq:q_err} and the upper bound $\Vert \nabla_{\theta_i}  \log \zeta_i^{\theta_i} (a_i|s_i)\Vert \leq L_i$. This concludes the proof of Lemma~\ref{lem:truncated_pg}.
\qed

\subsection{An Example of Encoding NP-Hard Problems into MARL Setup} \label{sec:nphard}
In this subsection, we provide an example on how NP hard problems can be encoded into the averge reward MARL problem with local interaction structure. We use the example of $k$-graph coloring in graph theory described as follows \citep{golovach2014survey}. Given a graph $\mathcal{G} = (\mathcal{V},\mathcal{E})$ and a set of $k$ colors $U$, a coloring is an assignment of a color $u\in U$ to each vertex in the graph, and a proper coloring is a coloring in which every two adjacent vertices have different colors. The $k$-graph coloring problem is to decide for a given graph, whether a proper coloring using $k$ colors exists, and is known to be NP hard when $k\geq 3$ \citep{golovach2014survey}.  

In what follows, we encode the $k$-coloring problem into our problem set up. %, and show that if we can maximize the average reward~\eqref{eq:average_reward} in our problem, we can also solve the $k$-coloring problem. 
Given a graph $\mathcal{G} = (\mathcal{V},\mathcal{E})$, we identify the set of agents $\mathcal{N}$ with $\mathcal{V}$ and their interaction graph as $\mathcal{G}$. The local state is tuple $s_i=(u_i,b_i) \in \mathcal{S}_i = U\times \{0,1\}$, where $u_i$ represents the color of node $i$ and $b_i$ is a binary variable indicating whether node $i$ has a different 
color from all nodes adjacent to $i$. We also identify action space $\mathcal{A}_i = U$ to be the set of colors. The state transition is given by the following rule, which satisfy the local interaction structure in \eqref{eq:transition_factor}: given $s_j(t) = (u_j(t),b_j(t))$ for $j\in N_i$ and $a_i(t)$, we set $u_i(t+1) = a_i(t)$, and if for all neighbors $j\in N_i/\{i\}$, $u_i(t)$ does not have the same color as $u_j(t)$, we set $b_i(t+1) = 1$; otherwise, set $b_i(t+1) = 0$. 

Given $s_i = (u_i,b_i)$ and $a_i$, we also set the local reward $r_i(s_i,a_i) = 1$ if $b_i = 1$ (i.e. node $i$ does not have the same color as any of its adjacent nodes), and otherwise the reward is set as $0$. The local policy class is such that $a_i(t)$ is not allowed to depend on $s_i(t)$ but can be drawn from any distribution on the action space, i.e. the set of colors. In other words, $\zeta_i^{\theta_i}(\cdot)$ is a distribution on the action space, parameterized by $\theta_i$.  

For policy $\theta = (\theta_i)_{i\in\mathcal{V}}$, it is clear that the stationary distribution of $s_i$ is simply $\zeta_i^{\theta_i}$; the stationary distribution for $b_i$, which we denote as $\pi_{b_i}^{\theta}$, is given by, 
$ \pi_{b_i}^{\theta} (1) =  \PR(a_i\neq a_j, \forall j\in N_i/\{i\})$, where in the probability, $a_i$ is independently sampled from $\zeta_i^{\theta_i}$ and $a_j$ from $\zeta_j^{\theta_j}$. Further, in this case, the objective function (average reward) is given by 
\[J(\theta) =  \frac{1}{|\mathcal{V}|} \sum_{i\in\mathcal{V}} \pi_{b_i}^{\theta} (1) .\]
%where $ s^{\textrm{init}}$ is an initial state where $s^{\textrm{init}}_i = (u^{\textrm{init}}, 0)$ with $u^{\textrm{init}}\in U$ being a fixed color. 
%This objective function is another form of definition for average reward \citep{bertsekas1996neuro}, and is equivalent to the original definition~\eqref{eq:average_reward} when the MDP is ergodic under policy $\theta$. 
It is immediately clear that in the above set up, the maximum possible average reward is $1$ if and only if there exists a proper coloring in the $k$-coloring problem. 
To see this, if there exists a proper coloring $(u_i^*)_{i\in\mathcal{V}}$ in the $k$-coloring problem, then a policy that always sets $a_i(t) = u_i^*$ will drive $s_i(t)$ in two steps to a fixed state $ s_i = (u_i^*, 1)$, which will result in average reward $1$. On the contrary, if there exists a policy achieving average reward $1$, then the support of the action distribution in the policy constitute a set of proper colorings. %This is because 

As such, if we can maximize the average reward, then we can also solve the $k$-coloring problem, which is known to be NP-hard when $k\geq 3$. This highlights the difficulty of the average reward MARL problem. 

\section{The exponential decay property and proof of Theorem~\ref{thm:exp_decay}}
In this section, we formally prove Theorem~\ref{thm:exp_decay} in \Cref{sec:proof_exp_decay} that bounded interaction guarantees the exponential decay property holds. We will also provide a proof of Corollary~\ref{cor:discounted} in \Cref{subsec:discounted}, and provide numerical validations of the exponential decay property in \Cref{sec:example_exp_decay}. 

\subsection{Proof of Theorem~\ref{thm:exp_decay}}\label{sec:proof_exp_decay}
Set $s = (s_{\nik},s_{\nminusik}), a= (a_{\nik}, a_{\nminusik})$, and $\tilde{s} = (s_{\nik},\tilde{s}_{\nminusik})$, $\tilde{a} =( a_{\nik}, \tilde{a}_{\nminusik}) $. Recall the exponential decay property (Definition~\ref{def:exp_decay}) is a bound on
\begin{align*}
&\Big |Q_i^\theta(s,a) - Q_i^\theta(\tilde{s},\tilde{a}) \Big| \\
&\leq \sum_{t=0}^\infty \Big|\E_\theta [ r_i(s_i(t),a_i(t))| s(0) = s, a(0) = a] - \E_\theta [ r_i(s_i(t),a_i(t))| s(0) = s', a(0) = a'] \Big| \\
%&\leq \sum_{t=0}^\infty \TV(\pi^\theta_{t,i}\zeta_i^{\theta_i}, \pi'^\theta_{t,i}\zeta_i^{\theta_i}) \bar{r}\\
&\leq  \sum_{t=0}^\infty \TV(\pi^\theta_{t,i}, \tilde\pi^\theta_{t,i}) \bar{r},
\end{align*}
where $\pi^\theta_{t,i}$ means the distribution of $(s_i(t),a_i(t))$ conditioned on $(s(0),a(0)) = (s,a)$, and similarly $\tilde\pi^\theta_{t,i}$ is the distribution of $(s_i(t),a_i(t))$ conditioned on $(s(0),a(0)) = (\tilde{s},\tilde{a})$. It is immediately clear that, $\pi^\theta_{t,i}= \tilde\pi^\theta_{t,i}$ for $t\leq \khop$. Therefore, if we can show that
\begin{align}
    \TV(\pi^\theta_{t,i}, \tilde\pi^\theta_{t,i})\leq  \rho^t \text{ for } t>\khop, \label{eq:exp_decay:pi_difference}
\end{align}
it immediately follows that
\[|Q_i^\theta(s,a) - Q_i^\theta(\tilde{s},\tilde{a})|\leq \sum_{t=\khop+1}^\infty  \rho^t\bar{r} = \frac{ \bar{r}}{1-\rho}\rhok, \]
which is the desired exponential decay property. 

We now show \eqref{eq:exp_decay:pi_difference}. Our primary tool is the following result on Markov chain with product state spaces, whose proof is deferred to \Cref{subsec:proof_markov_exp_decay}.
\begin{lemma}\label{lem:markov_chain_exp_decay}
Consider a Markov Chain with state $z=(z_1,\ldots,z_n)\in \mathcal{Z}=\mathcal{Z}_1\times\cdots\times \mathcal{Z}_n$, where each $\mathcal{Z}_i$ is some finite set. Suppose its transition probability factorizes as
\[ P(z(t+1)|z(t)) = \prod_{i=1}^n P_i(z_i(t+1)|z_{N_i}(t))\]
and further, if $\sup_{1\leq i\leq n}\sum_{j=1}^nC^z_{ij}\leq \rho $, where
\[C_{ij}^z =  \left\{ \begin{array}{ll}
0, & \text{ if }j\notin N_i,\\
\sup_{z_{N_i/j}}\sup_{z_j,z_j'} \TV(  P_i(\cdot|z_j, z_{N_i/j}) , P_i(\cdot| z_j', z_{N_i/j})   ) , & \text{ if } j\in N_i,
\end{array} \right.  \]
then for any $z = (z_{\nik},z_{\nminusik})$, $\tilde{z}=(z_{\nik},\tilde{z}_{\nminusik})$, we have,
\[\TV(\pi_{t,i},\tilde{\pi}_{t,i}) = 0 \quad \text{for}\quad t\leq \khop,\qquad \TV(\pi_{t,i},\tilde{\pi}_{t,i}) \leq  \rho^t \quad \text{for}\quad t>\khop,\]
where $\pi_{t,i}$ is the distribution of $z_i(t)$ given $z(0)=z$, and $\tilde{\pi}_{t,i}$ is the distribution of $z_i(t)$ given $z(0) = \tilde{z}$. 
\end{lemma}
We now set the Markov chain in Lemma~\ref{lem:markov_chain_exp_decay} to be the induced Markov chain of our MDP with a localized policy $\theta$, with $z_i=(s_i,a_i)$ and $\mathcal{Z}_i = \mathcal{S}_i\times \mathcal{A}_i$. For this induced chain, we have the transition factorized as,
\[P(s(t+1),a(t+1)|s(t),a(t)) = \prod_{i=1}^n \zeta_i^{\theta_i}(a_i(t+1)|s_i(t+1))P_i(s_i(t+1)|s_i(t),a_i(t),s_{N_i}(t)). \]
Then, $C_{ij}^z$ in Lemma~\ref{lem:markov_chain_exp_decay} becomes
\[C_{ij}^z = \left\{ \begin{array}{ll}
0, & \text{ if } j\notin N_i,\\
\sup_{s_{N_i/j},a_i}\sup_{s_j,s_j'} \TV(  P_i(\cdot|s_j, s_{N_i/j},a_i) , P_i(\cdot| s_j', s_{N_i/j},a_i)   ),  & \text{ if } j\in N_i/i, \\
\sup_{s_{N_i/i}}\sup_{s_i,s_i',a_i,a_i'} \TV(  P_i(\cdot|s_i,a_i, s_{N_i/i}) , P_i(\cdot| s_i',a_i', s_{N_i/i})   ),  & \text{ if } j=i,
\end{array} \right.\]
which is precisely the definition of $C_{ij}$ in Theorem~\ref{thm:exp_decay}. As a result, the condition in Theorem~\ref{thm:exp_decay} implies the condition in Lemma~\ref{lem:markov_chain_exp_decay} holds, regardless of the policy parameter $\theta$. Therefore, \eqref{eq:exp_decay:pi_difference} holds and Theorem~\ref{thm:exp_decay} is proven. 

\subsection{Proof of Lemma~\ref{lem:markov_chain_exp_decay}}\label{subsec:proof_markov_exp_decay}
%We first define a set of notations for the proof. 
%We perturb local transition probabilities $\{P_j\}$ to $\{\tilde{P}_j\}$, and let $b_j = \sup_{s_{N_j}} \TV(P_j(\cdot| s_{N_j}), \tilde{P}_j(\cdot |s_{N_j}))$. Clearly $b_j = 0$ if $P_j = \tilde{P}_j$ i.e. $P_j$ is unperturbed. Also, $b_j\in [0,1]$. Here we allow $b_j$ to be nonzero for any $j$, but we note that in the definition of type I exponential decaying property, $b_j \neq 0$ only for $j\in N_i^{k+1}/\nik$; in the definition of type II exponential decaying property, $b_j=0$ for all $j$ (no perturbation at all).

We do two runs of the Markov chain, one starting with $z$ with trajectory $z(0),\ldots,z(t),\ldots$, and another starting with $\tilde{z}$ with trajectory $\tilde{z}(0),\ldots,\tilde{z}(t),\ldots$  We use $\pi_{t}$ ($\tilde{\pi}_t$) to denote the distribution of $z(t)$ ($\tilde{z}(t)$); $\pi_{t,i}$ ($\tilde{\pi}_{t,i}$) to be the distribution of $z_i(t)$ ($\tilde{z}_i(t)$), $\pi_{t,\nik}$ ($\tilde{\pi}_{t,\nik}$) to denote the distribution of $z_{\nik}(t)$ ($\tilde{z}_{\nik}(t)$).

%\begin{theorem}
%	Suppose the matrix $C= [C_{ij}]$ satisfies $\sup_i \sum_{j}C_{ij} \leq \rho<1$. Then, fixing $i$, if $b_j=0$ for $j\in \nik$, i.e. the $P_j$ is only perturbed outside of $\nik$. Then, for any $t$, 
%	$$\TV(P_{t,i},\tilde{P}_{t,i}) \leq \frac{\rho^{k+1}}{1-\rho}$$
%	Letting $t\rightarrow \infty$, this shows that the marginalized stationary distribution at node $i$ changes by at most $\frac{\rho^{k+1}}{1-\rho}$ (in \TV distance) after perturbation at nodes outside of $\nik$. This also shows the ``truncation'' idea affect the marginalized stationary distribution at node $i$ by at most $\frac{\rho^{k+1}}{1-\rho}$, because the truncation idea is simply a perturbation of $P_j$ to be the uniform distribution at boundary nodes $j\in N_i^{k+1}/\nik$ . 
%\end{theorem}

\smallskip
Our proof essentially relies on induction on $t$, and the following Lemma is the key step in the induction. 

\begin{lemma} \label{lem:a_induction}
Given $t$, we say $a = [a_1,\ldots,a_n]^\top$ is $(t-1)$-compatible if for any $i,\khop$, and for any function $f: \R^{\mathcal{Z}_{\nik}}\rightarrow\R$, 
\[ | \E_{z_{\nik}\sim \pi_{t-1,\nik} } f(z_{\nik}) - \E_{z_{\nik}\sim \tilde \pi_{t-1,\nik} } f(z_{\nik})  | \leq \sum_{j\in \nik}  a_j \delta_j(f), \]
	where $\delta_j(f) $ is the variation of $f$'s dependence on $z_j$, i.e. $\delta_j(f)= \sup_{z_{\nik/j}} \sup_{z_j, z_j'} |f(z_j,z_{\nik/j}  ) -f(z_j',z_{\nik/j}  )  | $. Suppose now that $a $ is $(t-1)$-compatible, then we have $a' = [a_1',\ldots,a_n']^\top$ is $t$-compatible, with $a' = Ca $, where $C\in\R^{n\times n}$ is the matrix of $[C_{ij}^z]$. %for any $i,k$ and $f: \R^{\mathcal{S}_{\nik}}\rightarrow\R$ 
%	$$ | \E_{z_{\nik}\sim \pi_{t,\nik} } f(z_{\nik}) - \E_{z_{\nik}\sim \tilde \pi_{t,\nik} } f(z_{\nik})  | \leq \sum_{\ell\in N_{i}^{k}}  a_\ell' \delta_\ell(f) $$

\end{lemma}

Now we use Lemma~\ref{lem:a_induction} to prove \Cref{lem:markov_chain_exp_decay}. We fix $i$ and $\khop$. Since $z_{\nik}(0) = \tilde{z}_{\nik}(0)$, we can set $a_j^{(0)} = 0$ for $j\in \nik$, and  $a_j^{(0)} = 1$ for $j\not\in \nik$. It is easy to check such $a^{(0)}$ is $0$-compatible. 
As a result, $a^{(t)} = C^t a^{(0)}$ is $t$-compatible. Since $a^{(0)}$ is supported outside $\nik$, we have for all $t\leq \khop$, $a_i^{(t)} = 0$; and for $t\geq \khop+1$, $a_i^{(t)} = [C^t a^{(0)}]_i \leq (\Vert C\Vert_\infty)^t \Vert a^{(0)}\Vert_\infty \leq \rho^t$. 
As a result, by the definition of $t$-compatible, we set $f:\R^{\mathcal{Z}_i}\rightarrow\R$ to be the indicator function for any event $A_i\subset \mathcal{Z}_i$ (i.e. $f (z_i) = \mathbf{1}(z_i\in A_i)$) and get,
\[ | \pi_{t,i}(z_i \in A_i) - \tilde{\pi}_{t,i}(z_i \in A_i) | \leq a_i^{(t)} ,\]
and if we take the $\sup$ over $A_i$, we directly get,
\[ \TV(\pi_{t,i}, \tilde{\pi}_{t,i})\leq a_i^{(t)}\leq \rho^t, \]
which finishes the proof of Lemma~\ref{lem:markov_chain_exp_decay}. It remains to prove the induction step Lemma~\ref{lem:a_induction}, which is done below.

\textit{Proof of Lemma~\ref{lem:a_induction}:} Recall that the transition probability can be factorized as follows,
\[ P(z(t+1)|z(t)) = \prod_{i=1}^n P_i(z_i(t+1)|z_{N_i}(t)),\]
where the distribution of $z_i(t+1)$ only depends on $z_{N_i}(t)$ with transition probability given by $P_i(z_i(t+1)|z_{N_i}(t))$. We also define $P_{N_i^k}$ to be the transition from $z_{N_i^{k+1}}(t)$ to $z_{N_i^k}(t+1)$,
\[P_{N_i^k} (z_{N_i^k}(t+1)|z_{N_i^{k+1}}(t) ) = \prod_{j\in N_i^k} P_j(z_j(t+1)|z_{N_j}(t)). \]
With these definitions, we have for any $i,\khop$, and for any function $f: \R^{\mathcal{Z}_{\nik}}\rightarrow\R$,
	\begin{align}
&\bigg| \E_{z_{\nik}\sim \pi_{t,\nik} } f(z_{\nik}) - \E_{z_{\nik}\sim \tilde \pi_{t,\nik} } f_{\nik}(z_{\nik})  \bigg| = \nonumber\\
&=\bigg| \E_{z_{N_i^{\khop+1}}'\sim \pi_{t-1,N_i^{\khop+1}} } \E_{z_{\nik} \sim P_{\nik}(\cdot|z_{N_i^{\khop+1}}') } f(z_{\nik}) - \E_{z_{N_i^{\khop+1}}'\sim \tilde \pi_{t-1,N_i^{\khop+1}} } \E_{z_{\nik} \sim  P_{\nik}(\cdot|z_{N_i^{\khop+1}}') } f(z_{\nik}) \bigg|\nonumber\\
&=  \bigg| \E_{z_{N_i^{\khop+1}}'\sim \pi_{t-1,N_i^{\khop+1}} } g(z_{N_i^{\khop+1}}') - \E_{z_{N_i^{\khop+1}}'\sim \tilde \pi_{t-1,N_i^{\khop+1}} } g(z_{N_i^{\khop+1}}')  \bigg|, \label{eq:a_induction:eq_1}
	\end{align}
	where we have defined $ g(z_{N_i^{\khop+1}}')  = \E_{z_{\nik} \sim P_{\nik}(\cdot|z_{N_i^{\khop+1}}') } f(z_{\nik}) $. Since $a=[a_1,\ldots,a_n]^\top$ is $(t-1)$ compatible, we have,
\begin{align*}
 \bigg| \E_{z_{N_i^{\khop+1}}'\sim \pi_{t-1,N_i^{\khop+1}} } g(z_{N_i^{\khop+1}}') - \E_{z_{N_i^{\khop+1}}'\sim \tilde \pi_{t-1,N_i^{\khop+1}} } g(z_{N_i^{\khop+1}}')  \bigg| \leq \sum_{j\in N_i^{\khop+1}} a_j \delta_j( g).
\end{align*}
Now we analyze $\delta_j(g)$. We fix $z_{N_i^{\khop+1}/j}'$, then
\begin{align*}
g(z_j', z_{N_i^{\khop+1}/j}') - g(z_j'',z_{N_i^{\khop+1}/j}')  = \E_{z_{\nik} \sim P_{\nik}(\cdot|z_j' ,z_{N_i^{\khop+1}/j}') } f(z_{\nik}) - \E_{z_{\nik} \sim P_{\nik}(\cdot|z_j'' ,z_{N_i^{\khop+1}/j}') } f(z_{\nik}).
\end{align*}
Taking a closer look,  both $P_{\nik}(\cdot|z_j' ,z_{N_i^{\khop+1}/j}')    $ and  $P_{\nik}(\cdot|z_j'' ,z_{N_i^{\khop+1}/j}')    $ are product distributions on the states in $\nik$, and they differ only for those $\ell\in \nik$ that are adjacent to $j$, i.e. $\nik\cap N_j$. Therefore, we can use the following auxiliary result whose proof is provided in the bottem of this subsection. 
\begin{lemma}\label{lem:diff_dist_f}
	For a function $f$ that depends on a group of variables $z = (z_i)_{i\in V}$, let $P_i$ and $\tilde{P_i}$ to be two distributions on $z_i$. Let $P$ be the product distribution of $P_i$ and $\tilde{P}$ be the product distribution of
	 $\tilde{P}_i$. Then
	 \[|\E_{z\sim P} f(z) - \E_{z\sim \tilde{P}} f(z) | \leq \sum_{i\in V} \TV(P_i, \tilde{P}_i) \delta_i (f) .\]
\end{lemma}

By Lemma~\ref{lem:diff_dist_f}, we have,
\begin{align*}
|g(z_j', z_{N_i^{\khop+1}/j}') - g(z_j'',z_{N_i^{\khop+1}/j}')|  &\leq \sum_{\ell \in \nik\cap N_j} \TV(P_\ell(\cdot|z_j', z_{N_\ell/j}' ), P_\ell(\cdot|z_j'', z_{N_\ell/j}' )) \delta_\ell(f) \\
&\leq \sum_{\ell \in \nik\cap N_j} C_{\ell j}^z\delta_\ell(f) .
\end{align*}
As such, $\delta_j(g) \leq \sum_{\ell \in \nik\cap N_j} C_{\ell j}^z \delta_\ell(f) $, and we can continue \eqref{eq:a_induction:eq_1} and get,
\begin{align*}
\bigg| \E_{z_{\nik}\sim \pi_{t,\nik} } f(z_{\nik}) - \E_{z_{\nik}\sim \tilde \pi_{t,\nik} } f_{\nik}(z_{\nik})  \bigg| &\leq \sum_{j\in N_i^{\khop+1}} a_j \delta_j(g) \\%+ \sum_{j\in N_{i}^{k}} b_j \delta_j(f) \\
&\leq \sum_{j\in N_i^{\khop+1}} a_j \sum_{\ell \in \nik\cap N_j} C_{\ell j}^z\delta_\ell(f) \\% + \sum_{j\in N_{i}^{k}} b_j \delta_j(f)\\
&= \sum_{\ell \in \nik} \sum_{j\in N_\ell} a_j  C_{\ell j}^z\delta_\ell(f). %+ \sum_{j\in N_{i}^{k}} b_j \delta_j(f)\\
%&= \sum_{\ell \in \nik}  [\sum_{j\in N_\ell} a_j  C_{\ell j}  + b_\ell ]\delta_\ell(f)
\end{align*}
This implies $a' = [a_1',\ldots,a_n']^\top$ is $t$-compatible, where $a_\ell' = \sum_{j\in N_\ell} C_{\ell j}^z a_j $. \qed

Finally, we provide the proof of Lemma~\ref{lem:diff_dist_f}.

\textit{Proof of Lemma~\ref{lem:diff_dist_f}:}	We do induction on the size of $|V|$. For $|V|=1$, we have 
		 \[|\E_{z_1\sim P_1 } f(z_1) - \E_{z_1\sim \tilde{P}_1 } f(z_1) | =| \langle P_1, f\rangle - \langle \tilde{P}_1, f\rangle |,\] 
where both $P_1,\tilde{P}_1$ and $f$ are interpreted as vectors indexed by $z_1$, and $\langle\cdot,\cdot\rangle$ is the usual inner product. Let $\mathbf{1}$ be the all one vector with the same dimension of $P_1, \tilde{P}_1$ and $f$. Let $m$ and $M$ be the minimum and maximum value of $f$ respectively. Then, 
		 \begin{align*}
		 	&	 | \langle P_1, f\rangle - \langle \tilde{P}_1, f\rangle | =  | \langle P_1 - \tilde{P}_1, f - \frac{M+m}{2}\mathbf{1}\rangle | \\
		 		 &\leq \Vert P_1 - \tilde{P}_1\Vert_1 \Vert f - \frac{M+m}{2}\mathbf{1}\Vert_\infty = \frac{M-m}{2} \Vert P_1 - \tilde{P}_1\Vert_1 = \TV(P_1,\tilde{P}_1) \delta_1 (f) .
		 \end{align*}
As a result, the statement is true for $|V|=1$. Suppose the statement is true for $|V|=n-1$. Then, for $|V|=n$, we use $z_{2:n}$ to denote $(z_2,\ldots,z_n)$ and use $P_{2:n}$ to denote the product distribution $P_{2:n}(z_2,\ldots, z_n) = \prod_{i=2}^n P_i (z_i)$; $\tilde{P}_{2:n}$ is defined similarly. Then,
\begin{align*}
|\E_{z\sim P} f(z) - \E_{z\sim \tilde{P}} f(z) | &= |\E_{z_1\sim P_1} \E_{z_{2:n}\sim P_{2:n}} f(z_1, z_{2:n}) - \E_{z_1\sim \tilde{P}_1} \E_{z_{2:n}\sim \tilde P_{2:n}} f(z_1, z_{2:n}) | \\
&\leq  |\E_{z_1\sim P_1} \E_{z_{2:n}\sim P_{2:n}} f(z_1, z_{2:n}) - \E_{z_1\sim {P}_1} \E_{z_{2:n}\sim \tilde P_{2:n}} f(z_1, z_{2:n}) |\\
&\quad + |\E_{z_1\sim P_1} \E_{z_{2:n}\sim \tilde P_{2:n}} f(z_1, z_{2:n}) - \E_{z_1\sim \tilde{P}_1} \E_{z_{2:n}\sim \tilde P_{2:n}} f(z_1, z_{2:n}) |\\
&\leq  \E_{z_1\sim P_1} |\E_{z_{2:n}\sim P_{2:n}} f(z_1, z_{2:n}) -  \E_{z_{2:n}\sim \tilde P_{2:n}} f(z_1, z_{2:n}) |\\
&\quad + |\E_{z_1\sim P_1} \bar{f}(z_1)- \E_{z_1\sim \tilde{P}_1} \bar{f}(z_1) |,
\end{align*}
where we have defined $\bar{f}(z_1) = \E_{z_{2:n}\sim \tilde P_{2:n}} f(z_1, z_{2:n})   $. Fixing $z_1$, we have by induction assumption, 
\[ |\E_{z_{2:n}\sim P_{2:n}} f(z_1, z_{2:n}) - \E_{z_{2:n}\sim \tilde P_{2:n}} f(z_1, z_{2:n}) | \leq \sum_{i=2}^n \TV(P_i,\tilde{P}_i) \delta_i(f(z_1,\cdot)) \leq\sum_{i=2}^n \TV(P_i,\tilde{P}_i) \delta_i(f) .   \]
Further, we have,
\begin{align*}
\delta_1 (\bar{f}) &= \sup_{z_1,z_1'}  |\E_{z_{2:n}\sim \tilde P_{2:n}} f(z_1, z_{2:n})  - \E_{z_{2:n}\sim \tilde P_{2:n}} f(z_1', z_{2:n}) | \\
& \leq  \sup_{z_1,z_1'}  \E_{z_{2:n}\sim \tilde P_{2:n}} |f(z_1, z_{2:n})  -f(z_1', z_{2:n}) |  \\
&\leq \sup_{z_1,z_1'}  \sup_{z_{2:n}} |f(z_1, z_{2:n})  -f(z_1', z_{2:n}) | = \delta_1(f).
\end{align*}
Combining these results, we have
\begin{align*}
|\E_{s\sim P} f(s) - \E_{s\sim \tilde{P}} f(s) | 
&\leq  \E_{z_1\sim P_1} \sum_{i=2}^n \TV(P_i,\tilde{P}_i) \delta_i(f)     + \TV(P_1,\tilde{P}_1) \delta_1(\bar{f})\\
&\leq \sum_{i=1}^n \TV(P_i,\tilde{P}_i) \delta_i(f)  .
\end{align*}
So the induction is finished and the proof of Lemma~\ref{lem:diff_dist_f} is concluded.
\qed
\subsection{Proof of Corollary~\ref{cor:discounted}}\label{subsec:discounted} 
In the $\gamma$-discounted case, the $Q$-function is defined as \citep{qu2019scalable}, 
\begin{align}
    Q_i^\theta(s,a) =  \E_{a(t) \sim \zeta^{\theta}(\cdot|s(t))} \bigg[ \sum_{t=0}^\infty \gamma^t r_i(s_i(t),a_i(t) ) \bigg|s(0) = s,a(0) = a \bigg]. \label{eq:discounted_q}
\end{align}

For notational simplicity, denote $s = (s_{\nik},s_{\nminusik})$, $a = (a_{\nik}, a_{\nminusik})$; $s'= (s_{\nik},s_{\nminusik}')$ and $a' = (a_{\nik}, a_{\nminusik}')$. Let $\pi_{t,i}$ be the distribution of $(s_i(t),a_i(t))$ conditioned on $(s(0),a(0)) = (s,a)$ under policy $\theta$, and let $\pi_{t,i}'$ be the distribution of $(s_i(t),a_i(t))$ conditioned on $(s(0),a(0))=(s',a')$ under policy $\theta$. Then, under the conditions of \Cref{thm:exp_decay}, we can use equation~\eqref{eq:exp_decay:pi_difference} in the proof of Theorem~\ref{thm:exp_decay} (also see Lemma~\ref{lem:markov_chain_exp_decay}), which still holds in the discounted setting as equation~\eqref{eq:exp_decay:pi_difference} is a property of the underlying Markov chain, irrespective of how the objective is defined. This leads to,
\[\TV(\pi_{t,i},\pi_{t,i}') = 0 \quad \text{for}\quad t\leq \khop,\qquad \TV(\pi_{t,i},\pi_{t,i}') \leq  \rho^t \quad \text{for}\quad t>\khop.\]
%Then, we must have $\pi_{t,i} = \pi_{t,i}'$ for all $t\leq \khop$. The reason is that, due to the local dependence structure \eqref{eq:transition_factor} and the localized policy structure, $\pi_{t,i}$ only depends on $(s_{N_i^t},a_{N_i^t})$ (the initial state-action of agent $i$'th $t$-hop neighborhood) which is the same as $(s_{N_i^t}',s_{N_i^t}')$ when $t\leq \khop$ per the way the initial state $(s,a)$, $(s',a')$ are chosen. 

With these preparations, we verify the exponential decay property. We expand the definition of $Q_i^\theta$ in \eqref{eq:discounted_q},
\begin{align*}
   &|Q_i^\theta(s,a) - Q_i^\theta(s',a')| \nonumber \\ &\leq   \sum_{t=0}^\infty \bigg| \E \big[\gamma^t r_i(s_i(t),a_i(t))  \big|(s(0),a(0)) = (s,a) \big] - \E \big[\gamma^t r_i(s_i(t),a_i(t))  \big|(s(0),a(0)) = (s',a') \big] \bigg| \nonumber \\
  & = \sum_{t=0}^\infty \bigg| \gamma^t \E_{(s_i,a_i)\sim \pi_{t,i}} r_i(s_i,a_i)  -\gamma^t \E_{(s_i,a_i)\sim \pi_{t,i}'} r_i(s_i,a_i) \bigg|\nonumber \\
  &\leq \sum_{t=0}^\infty \gamma^t   \bar{r} \text{TV}( \pi_{t,i},\pi_{t,i}')   \leq \sum_{t=\khop+1}^\infty  \gamma^t   \bar{r} \rho^t \leq  \frac{\bar{r}}{1-\gamma\rho } (\gamma\rho)^{\khop+1}.
\end{align*}
%where $ \text{TV}( \pi_{t,i},\pi_{t,i}') $ is the total variation distance between $\pi_{t,i}$ and $\pi_{t,i}'$ which is upper bounded by $1$. 
The above inequality shows that the $(\frac{\bar{r}}{1-\rho\gamma},\rho\gamma)$-exponential decay property holds and concludes the proof of Corollary~\ref{cor:discounted}. 

\subsection{Numerical Validation of the Exponential Decay Property}\label{sec:example_exp_decay}
In this subsection, we conduct numerical experiments to show that the exponential decay property holds broadly for randomly generated problem instances.

We consider a line graph with $n = 10$ nodes, local state space size $|\mathcal{S}_i| = 2$, local action space size $|\mathcal{A}_i| =3$. 
We generate the local transition probabilities $P_i$, localized polices $\zeta_i$ and local rewards $r_i$ uniformly randomly with maximum reward set to be $1$. 
To verify the exponential decay property, we consider Definition~\ref{def:exp_decay}, where we pick $i$ to be the left most node in the line, generate $s, s',a,a'$ uniformly random in the global state or action space, and then increase $\khop$ from $0$ to $n-2$. 
For each $\khop$, we calculate the left hand side of \eqref{eq:exp_decay_def} exactly through brutal force. We repeat the above procedure 100 times, each time with a newly generated instance, and plot the left hand side of \eqref{eq:exp_decay_def} as a function of $\kappa$ in Figure~\ref{fig:exp_decay_numerical_line}. 

We do a similar experiment on a $2\times 6$ 2D grid, with a similar setup except node $i$ is now selected as the corner node in the grid. The results are shown in Figure~\ref{fig:exp_decay_numerical_grid}. Both Figure~\ref{fig:exp_decay_numerical_line} and Figure~\ref{fig:exp_decay_numerical_grid} confirm that the left hand side of \eqref{eq:exp_decay_def} decay exponentially in $\kappa$. This shows that the exponential decay property holds broadly for instances generated randomly. \lina{should we discuss the $\rho$?} \adam{I'm okay either way}

\begin{figure}
    \centering
    \begin{subfigure}[t]{0.5\textwidth}
        \includegraphics[width=\textwidth]{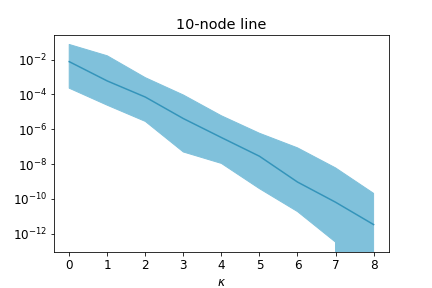}
        \caption{}\label{fig:exp_decay_numerical_line}
    \end{subfigure}~
    \begin{subfigure}[t]{0.5\textwidth}
        \includegraphics[width=\textwidth]{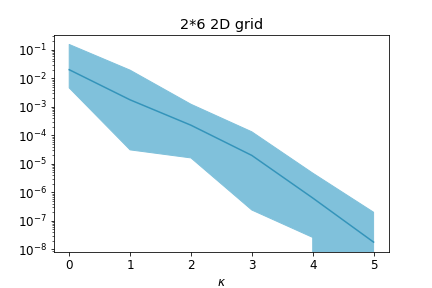}
         \caption{}\label{fig:exp_decay_numerical_grid}
    \end{subfigure}
    
    \caption{Numerical verification of the exponential decay property. The $y$-axis is the left hand side of~\eqref{eq:exp_decay_def} whereas the $x$-axis is $\kappa$. The solid line represents the median value of different runs, whereas the shaded region represents 10\% to 90\% percentile of the runs. }
    \label{fig:exp_decay_numerical}
\end{figure}

\section{Analysis of the Critic}\label{sec:critic}
% {\color{blue}
% List of notation convention inconsistencies.
% \begin{itemize}
%     \item where to put $t$. Currently, for the critic variables, it is in the superscript. For the policy parameter $\theta$ and for the state $z$, it is in parenthesis. For step size, it is in the subscript. Additionally, we need to figure out where to put $t$, in terms of the distribution of $z(t)$.
%     \item where to put $\theta$. All $\theta$ is in super script. 
%     \item for distributions, especially the distirbution of $z(t)$, the notation is in consistent.
% \end{itemize}
% }

The goal of the section is to analyze the critic update (line~\ref{algo:critic_1} and \ref{algo:critic_2} in Algorithm~\ref{algorithm:sac}). Our algorithm is a two-time scale algorithm, where the critic runs faster than the actor policy parameter $\theta(t)$. Therefore, in what follows, we show that the truncated $Q$-function in the critic $\hat{Q}_i^t$ ``tracks'' a quantity $\hat{Q}_i^{\theta(t)}$, which is the fixed point of the critic update when the policy is ``frozen'' at $\theta(t)$. Further, we show that this fixed point is a good approximation of the true $Q$ function $Q_i^{\theta(t)}$ for policy $\theta(t)$ 
because of the exponential decay property. The formal statement is given in Theorem~\ref{thm:critic}.

\begin{theorem}\label{thm:critic}
The following two statements are true. 
\begin{itemize}
    \item [(a)] For each $i$ and $\theta$, there exists $\hat{Q}_i^{\theta} \in\R^{\hat{\mathcal{Z}}_{\nik}}$ which is an approximation of the true $Q$ function in the sense that, there exists scalar $c_i^\theta$ that depends on $\theta$, such that
\begin{align}
    \sqrt{\E_{z\sim \sd^\theta} |\hat{Q}_i^\theta(z_{\nik}) + c_i^\theta - Q_i^\theta(z)|^2   } \leq \frac{c\rhok}{1-\mu_D},
\end{align}
where $\hat{Q}_i^\theta(\tilde{z}_{\nik})$ is understood as $0$. 
\item[(b)] For each $i$, almost surely $\sup_{t\geq 0} \Vert \hat{Q}_i^t\Vert_\infty <\infty $. Further, $\hat{Q}_i^t$ tracks $\hat{Q}_i^{\theta(t)}$ in the sense that almost surely, $\lim_{t\rightarrow\infty} \hat{Q}_i^t - \hat{Q}_i^{\theta(t)} = 0$. 
\end{itemize}
\end{theorem}

Our proof relies on the result on two-time scale stochastic approximation in \citet{konda2003linear}. In \Cref{sec:sa}, we review the result in \citet{konda2003linear} and in \Cref{subsec:proof_critic}, we provide the proof for Theorem~\ref{thm:critic}. 

\subsection{Review of A Stochastic Approximation Result}\label{sec:sa}
In this subsection, we review a result on two time-scale stochastic approximation in \cite{konda2003linear} which will be used in our proof for Theorem~\ref{thm:critic}. Consider the following iterative stochastic approximation scheme with iterate $x^t \in \R^{m}$,\footnote{Our stochastic approximation scheme \eqref{eq:stochastic_approximation} is slightly different from \citet{konda2003linear} in that in \citet{konda2003linear}, $h^{\theta(t)}(\cdot)$ and $G^{\theta(t)}(\cdot)$ depend on $z(t+1)$ instead of $z(t)$. This change is without loss of generality as we can group two states togethoer, i.e. $y(t) = (z(t-1),z(t))$ and write our algorithm in the form of \citet{konda2003linear}.}
\begin{subequations}\label{eq:stochastic_approximation}
\begin{align}
        x^{t+1} &= x^{t} + \alpha_t(h^{\theta(t)}(z(t)) - G^{\theta(t)}(z(t)) x^t + \xi^{t+1} x^t),  \\
    \theta(t+1) &= \theta(t) + \eta_t H^{t+1},
    \end{align}
\end{subequations}
where $z(t)$ is a stochastic process with finite state space $\mathcal{Z}$; $h^\theta(\cdot):\mathcal{Z}\rightarrow \R^m, G^\theta(\cdot):\mathcal{Z}\rightarrow \R^{m\times m}$ are vectors or matrices depending on both parameter $\theta$ as well as the state $z$; $\xi^{t+1} \in\R^{m\times m}$ and $H^{t+1}$ is some vector that drives the change of $\theta(t)$. 

In what follows, we state Assumption \ref{assump:sa1} to \ref{assump:sa6} used in \citet{konda2003linear}. Assumption~\ref{assump:sa1} is related to the summability of the step size $\alpha_t$.

\begin{assumption}\label{assump:sa1}
The step size is deterministic, nonincreasing, and satisfies $\sum_t \alpha_t = \infty, \sum_t \alpha_t^2 <\infty.$
\end{assumption}

Let $\mathcal{F}_t$ be the $\sigma$ algebra generated by $\{z(k), H^k, x^k,\theta(k)\}_{k\leq t}$. Assumption~\ref{assump:sa2} says that the stochastic process $z(t)$ is Markovian and is driven by a transition kernal that depends on $\theta(t)$.  
\begin{assumption}\label{assump:sa2} There exists a parameterized family of transition kernels $P^\theta$ on state space $\mathcal{Z}$ such that, for every $A\subset \mathcal{Z}$, $\PR(z(t+1)\in A|\mathcal{F}_t)= \PR(z(t+1)\in A|z(t),\theta(t)) = P^{\theta(t)}(z(t+1)\in\mathcal{A}|z(t))$.
\end{assumption}

Assumption~\ref{assump:sa3} is a technical assumption on the transition kernel $P^\theta$ as well as $h^\theta$, $G^\theta$. 
\begin{assumption}\label{assump:sa3}
For each $\theta$, there exists function $\bar{h}(\theta)\in\R^m, \bar{G}(\theta)\in \R^{m\times m}$, $\hat{h}^\theta: \mathcal{Z}\rightarrow\R^m, \hat{G}^\theta: \mathcal{Z}\rightarrow \R^{m\times m} $ that satisfy the following. \begin{itemize}
    \item [(a)] For all $z\in\mathcal{Z}$,
    \begin{align*}
        \hat{h}^\theta(z) &= h^\theta(z) - \bar{h}(\theta)   + [P^\theta \hat{h}^\theta](z),\\
        \hat{G}^\theta(z) &= G^\theta(z) - \bar{G}(\theta) + [P^\theta \hat{G}^\theta ](z),
    \end{align*}
    where $P^\theta \hat{h}^\theta$ is a map from $\mathcal{Z}$ to $\R^m$ given by $[P^\theta \hat{h}^\theta](z) = \E_{z'\sim P^\theta(\cdot|z)} \hat{h}^\theta(z')$; similarly, $P^\theta \hat{G}^\theta $ is given by  $[P^\theta \hat{G}^\theta](z) =\E_{z'\sim P^\theta(\cdot|z)} \hat{G}^\theta(z') $.
    \item[(b)] For some constant $C$, 
    $\max(\Vert \bar{h}(\theta)\Vert, \Vert \bar{G}(\theta)\Vert)\leq C$ for all $\theta$.
    \item [(c)] For any $d>0$, there exists $C_d>0$ such that
    $\sup_t \E \Vert f^{\theta(t)}(z(t))\Vert^d \leq C_d $
    where $f^\theta$ represents any of the functions $\hat{h}^\theta, h^\theta, \hat{G}^\theta, G^\theta$. 
    \item[(d)] For some constant $C>0$ and for all $\theta,\bar{\theta}$,
    \[\max(\Vert \bar{h}(\theta) - \bar{h}(\bar{\theta})\Vert, \Vert \bar{G}(\theta) - \bar{G}(\bar{\theta})\Vert) \leq C\Vert \theta - \bar{\theta}\Vert.\]
    \item[(e)] There exists a positive constant $C$ such that for each $z\in\mathcal{Z}$, 
    \[\Vert P^\theta f^\theta (z) - P^{\bar{\theta}} f^{\bar{\theta}}(z) \Vert \leq C \Vert \theta - \bar{\theta}\Vert, \]
    where $f^\theta$ is any of the function $\hat{h}^\theta$ and $\hat{G}^\theta$. 
\end{itemize}
\end{assumption}
The next Assumption~\ref{assump:sa4} is to ensure that $\theta(t)$ changes slowly by imposing a bound on $H^t$ and requiring step size $\eta_t$ to be much smaller than $\alpha_t$.
\begin{assumption} \label{assump:sa4}The process $H^t$ satisfies
$ \sup_t \E |H^t|^d<\infty $
for all $d$. Further, the sequence $\eta_t$ is deterministic and satisfies 
$ \sum_t \big(\frac{\eta_t}{\alpha_t}  \big)^d <\infty$
for some $d>0$.
\end{assumption}

Assumption~\ref{assump:sa5} says that the $\xi^{t}$ is a martingale difference sequence.
\begin{assumption}\label{assump:sa5} $\xi^t$ is an $m\times m$ matrix valued $\mathcal{F}_t$-martingale difference, with bounded momemnts, i.e. 
\[\E \xi^{t+1}| \mathcal{F}_t = 0,\quad \sup_{t} \E\Vert \xi^{t+1}\Vert^d <\infty,\]
for each $d>0$.
\end{assumption}

The final Assumption~\ref{assump:sa6} requires matrix $\bar{G}(\theta)$ to be uniformly positive definite. 
\begin{assumption}[Uniform Positive Definiteness]\label{assump:sa6} There exists $a>0$ s.t. for all $x\in\R^m$and $\theta$, we have
\[x^\top \bar{G}(\theta) x \geq a\Vert x\Vert^2.\]
\end{assumption}
With the above assumtions, \citet[Lem. 12, Thm. 7]{konda2003linear} shows that the following theorem holds.
\begin{theorem}[\cite{konda2003linear}]\label{thm:SA}
Under Assumption~\ref{assump:sa1}-\ref{assump:sa6}, with probability $1$, $\sup_{t\geq 0} \Vert x^t\Vert <\infty$ and 
\[\lim_{t\rightarrow\infty} \Vert x^t - \bar{G}(\theta(t))^{-1} \bar{h}(\theta(t)) \Vert = 0. \]
\end{theorem}
In the next subsection, we will use the stochastic approximation result here to provide a proof of Theorem~\ref{thm:critic}.

\subsection{Proof of Theorem~\ref{thm:critic}}\label{subsec:proof_critic}
In this subsection, we will write our algorithm in the form of the stochastic approximation scheme~\eqref{eq:stochastic_approximation} and provide a proof of Theorem~\ref{thm:critic}. Throughout the rest of the section, we fix $i\in\mathcal{N}$. 

Define $\mathbf{e}_{z_{\nik}}\in\R^{\hat{Z}_\nik} $ to be the unit vector in $\mathbb{R}^{\hat{\mathcal{Z}}_{\nik}}$ when $z_\nik\neq \tilde{z}_{\nik}$, and is the zero vector when $z_{\nik} = \tilde{z}_\nik$ (the dummy state-action pair). Then, one can check that the critic part of our algorithm (line~\ref{algo:critic_1} and \ref{algo:critic_2} in Algorithm~\ref{algorithm:sac}) can be rewritten as,
\begin{align}
\hat{\mu}_i^{t+1} &= \hat{\mu}_i^{t} +  \alpha_t [r_i(z_i(t)) - \hat{\mu}_i^{t}],\label{eq:critic_update_1}\\
\hat{Q}_i^{t+1} &= \hat{Q}_i^{t} + \alpha_t[r_i(z_i(t)) - \hat{\mu}_i^t + \mathbf{e}_{z_{\nik}(t+1)}^\top \hat{Q}_i^{t} - \mathbf{e}_{z_{\nik}(t)}^\top  \hat{Q}_i^{t}] \mathbf{e}_{z_{\nik}(t)} .\label{eq:critic_update_2}
\end{align}
When written in vector form, the above equation becomes
\[\left[ \begin{array}{c}
\hat{\mu}_i^{t+1} \\
\hat{Q}_i^{t+1} 
\end{array} \right] = \left[ \begin{array}{c}
\hat{\mu}_i^{t} \\
\hat{Q}_i^{t} 
\end{array} \right] + \alpha_{t} \bigg[ - \left[ \begin{array}{cc}
 1 & 0\\
 \mathbf{e}_{z_{\nik}(t)}  & \mathbf{e}_{z_{\nik}(t)} [  \mathbf{e}_{z_{\nik}(t)}^\top- \mathbf{e}_{z_{\nik}(t+1)}^\top  ]
    \end{array}\right] \left[ \begin{array}{c}
\hat{\mu}_i^{t} \\
\hat{Q}_i^{t} 
\end{array} \right]  + \left[\begin{array}{c}
 r_i(z_i(t))\\
\mathbf{e}_{z_{\nik}(t)}  r_i(z_i(t))
  \end{array}
  \right] \bigg].\]
We rescale the $\hat{\mu}_i^t$ coordinate by a factor of $\rescale$ for technical reasons to be clear later, and rewrite the above equation in an equivalent form, 
\[\left[ \begin{array}{c}
\rescale \hat{\mu}_i^{t+1} \\
\hat{Q}_i^{t+1} 
\end{array} \right] = \left[ \begin{array}{c}
\rescale \hat{\mu}_i^{t} \\
\hat{Q}_i^{t} 
\end{array} \right] + \alpha_{t} \bigg[ - \left[ \begin{array}{cc}
 1 & 0\\
\frac{1}{\rescale} \mathbf{e}_{z_{\nik}(t)}  & \mathbf{e}_{z_{\nik}(t)} [  \mathbf{e}_{z_{\nik}(t)}^\top- \mathbf{e}_{z_{\nik}(t+1)}^\top  ]
    \end{array}\right] \left[ \begin{array}{c}
\rescale \hat{\mu}_i^{t} \\
\hat{Q}_i^{t} 
\end{array} \right]  + \left[\begin{array}{c}
\rescale  r_i(z_i(t))\\
\mathbf{e}_{z_{\nik}(t)}  r_i(z_i(t))
  \end{array}
  \right] \bigg].\]
We define $x_i^t = [\rescale \hat{\mu}_i^t; \hat{Q}_i^t]$ and, 
\begin{align*}
    \tilde{G}_i(z,z') &= \left[ \begin{array}{cc}
 1 & 0\\
\frac{1}{\rescale} \mathbf{e}_{z_{\nik}}  & \mathbf{e}_{z_{\nik}} [  \mathbf{e}_{z_{\nik}}^\top- \mathbf{e}_{z_{\nik}'}^\top  ]
    \end{array}\right], \quad%= \left[ \begin{array}{cc}
%\tilde{G}_{11}(z,z') & \tilde{G}_{12}(z,z')\\
%\tilde{G}_{21}(z,z')  & \tilde{G}_{22}(z,z')
%    \end{array}\right]  ,\\
   h_i(z)  = \left[\begin{array}{c}
\rescale  r_i(z_i)\\
\mathbf{e}_{z_{\nik}}  r_i(z_i)
  \end{array}
  \right]. %= \left[\begin{array}{c}
%h_1(z)\\
%h_2(z)
%  \end{array}
%  \right]. % \mathbf{e}_{z_{\nik}} [r_i(z_i) - \mu_i]
\end{align*}
  With the above definitions, the critic update equation \eqref{eq:critic_update_1} and \eqref{eq:critic_update_2} can be rewritten as the following,
  \begin{align}
       x_i^{t+1} = x_i^t + \alpha_t\Big[-\tilde{G}_i(z(t),z(t+1)) x_i^t + h_i(z(t))\Big]. \label{eq:critic_update_vec}
  \end{align}
Let $P^{\theta}$ be the transition matrix and the state-action pair when the policy is $\theta$. Because at time $t$, the policy is $\theta(t)$, as such the transition matrix from $z(t)$ to $z(t+1)$ is $P^{\theta(t)}$. We define
\begin{align}
    G_i^{\theta}(z) = \E_{z'\sim P^\theta (\cdot|z) }\tilde{G}_i(z,z') = \left[ \begin{array}{cc}
 1 & 0\\
\frac{1}{\rescale} \mathbf{e}_{z_{\nik}}  & \mathbf{e}_{z_{\nik}} [   \mathbf{e}_{z_{\nik}}^\top- P^\theta(\cdot|z)\Phi_i ]
    \end{array} \right]   
\end{align}
where $P^\theta(\cdot|z)$ is understood as the $z$'th row of $P^\theta$ and is treated as a row vector. Also, we have defined $\Phi_i\in\R^{\mathcal{Z}\times \hat{\mathcal{Z}}_{\nik}}$ to be a matrix with each row indexed by $z\in \mathcal{Z}$ and each column indexed by $z_{\nik}' \in  \hat{\mathcal{Z}}_{\nik} $, and its entries are given by $\Phi_i(z,z_{\nik}')=1$ if $z_{\nik} =z_{\nik}'$ and  $\Phi_i(z,z_{\nik}')=0$ elsewhere. Then, \eqref{eq:critic_update_vec} can be rewritten as,
\begin{align}
    x_i^{t+1} = x_i^t + \alpha_t\Big[- G_i^{\theta(t)}(z(t))x_i^t + h_i(z(t)) + \underbrace{[ G_i^{\theta(t)}(z(t)) - \tilde{G}_i(z(t),z(t+1))  ]}_{:=\xi_i^{t+1}}x_i^t  \Big]. \label{eq:critic_update_vec_final}
\end{align}
This will correspond to the first equation in the stochastic approximation scheme \eqref{eq:stochastic_approximation} that we reviewed in \Cref{sec:sa}.
Further, the actor update can be written as,
\begin{align}
    \theta(t+1) = \theta(t) + \eta_t  \Gamma(\mathbf{\hat{Q}^t})  \hat{g}(t) . \label{eq:actor_update_vector}
\end{align}
with $    \hat{g}_i(t) =  \nabla_{\theta_i} \log \zeta_i^{\theta_i(t)}(a_i(t)|s_i(t)) \frac{1}{n} \sum_{j \in \nik}  \hat Q_j^t(z_{N_j}(t))$. We identify equation \eqref{eq:critic_update_vec_final} and \eqref{eq:actor_update_vector} with the stochastic approximation scheme in \eqref{eq:stochastic_approximation}, where $x_i^t, G_i^\theta, h_i,\xi_i^{t+1}, \Gamma(\mathbf{\hat{Q}^t})  \hat{g}(t)$ are identified with the $x^t, G^\theta, h^\theta, \xi^{t+1}, H^{t+1}$ in \eqref{eq:stochastic_approximation} respectively. In what follows, we will check all the assumptions (Assumption~\ref{assump:sa1} to \ref{assump:sa6}) in \Cref{sec:sa} and invoke Theorem~\ref{thm:SA}. 

To that end, we first define $\bar{G}_i(\theta), \bar{h}_i(\theta), \hat{G}_i^\theta(z), \hat{h}_i^\theta(z)$, which will be the solution to the Poisson equation in Assumption~\ref{assump:sa3}(a). Given $\theta$, recall the stationary distribution under policy $\theta$ is $\sd^\theta$ and matrix $\SD^\theta = \diag(\sd^\theta)$. We define,
\begin{align*}
    \bar{G}_i(\theta) &= \E_{z\sim \sd^\theta } G_i^\theta (z) 
    %&= \sum_{s,a}d(s,a) \mathbf{e}_{s_{\nik},a_{\nik}} [    P(\cdot|s,a)\Phi_i - \mathbf{e}_{s_{\nik},a_{\nik}}^T ] \nonumber\\
    = \left[ \begin{array}{cc}
 1 & 0\\
\frac{1}{\rescale} \Phi_i^\top \sd^\theta  & \Phi_i^\top \SD^\theta \big[   \Phi_i - P^\theta\Phi_i \big] 
    \end{array}\right],  \\
 \bar{h}_i(\theta)    &=\E_{z\sim \sd^\theta} h_i(z) = \left[ \begin{array}{c}
      \rescale  (\sd^\theta)^\top r_i  \\
       \Phi_i^\top \SD^\theta r_i 
 \end{array}\right]   ,
\end{align*}
where in the last line, $r_i$ is understood as a vector over the entire state-action space $\mathcal{Z}$, (though it only depends on $z_i$). We also define,
\begin{align*}
    \hat{G}_i^\theta(z) &= \E_\theta [\sum_{t=0}^\infty [G_i^{\theta}(z(t)) - \bar{G}_i(\theta)  ] |z(0)=z],\\
   \hat{h}_i^\theta(z)  &= \E_\theta [\sum_{t=0}^\infty [h_i(z(t)) - \bar{h}_i(\theta)  ] |z(0)=z].
\end{align*}
It is easy to check that the above definitions will be the solution to the Poisson equation in Assumption~\ref{assump:sa3}(a). 

We will now start to check all the assumptions. We will frequently use the following auxiliary lemma, which is an immediate consequence of Assumption \ref{assump:ergodicity}. 

\begin{lemma} \label{lem:mixing}
Under Assumption~\ref{assump:ergodicity}, for vector $d\in\R^{\mathcal{Z}}$ such that $\mathbf{1}^\top d = 0$, we have, $\Vert ((P^\theta)^\top)^t d \Vert_1 \leq c_\infty \mu_D^t \Vert d \Vert_1$ for $c_\infty = \sqrt{\frac{|\mathcal{Z}|}{\sigma}}$.  
\end{lemma}
\begin{proof} As $P^\theta$ is a ergodic stochastic matrix with stationary distribution $\sd^\theta$, we have
    $(P^\theta - \mathbf{1} (\sd^\theta)^\top )^t =  (P^\theta)^t -\mathbf{1} (\sd^\theta)^\top  $. As a result, 
    \begin{align*}
        ((P^\theta)^\top)^t d = [((P^\theta)^\top)^t -  \sd^\theta \mathbf{1}^\top ] d = [ (P^\theta)^t - \mathbf{1} (\sd^\theta)^\top]^\top d = [ (P^\theta - \mathbf{1} (\sd^\theta)^\top)^t]^\top d.
    \end{align*}
    As a result, by Assumption~\ref{assump:ergodicity}, $\Vert ((P^\theta)^\top)^t d \Vert_{\SD^\theta} \leq \Vert [ P^\theta - \mathbf{1} (\sd^\theta)^\top]^\top\Vert_{\SD^\theta}^t \Vert d\Vert_{\SD^\theta} \leq \mu_D^t \Vert d\Vert_{\SD^\theta} $. The rest follows from a change of norm as $\sqrt{\frac{\sigma}{|\mathcal{Z}|}}\Vert d\Vert_1 \leq \sqrt{\sigma} \Vert d\Vert_2 \leq  \Vert d\Vert_{\SD^\theta} \leq \Vert d\Vert_2 \leq  \Vert d\Vert_1. $
%     \begin{align*}
%         \Vert d(t) - d'(t)\Vert_1& = \Vert ((P^\theta)^\top)^t (d(0) - d'(0))  \Vert_1 \nonumber \\
%     &= \Vert ((P^\theta)^t - \mathbf{1}(\sd^\theta)^\top)^\top (d(0) - d'(0))\Vert_1 \nonumber \\
%     &\leq \Vert (P^\theta)^t - \mathbf{1}(\sd^\theta)^\top\Vert_\infty\Vert d(0) - d'(0)\Vert_1\leq c_\infty\mu_D^t \Vert d(0) - d'(0)\Vert_1.
% \end{align*}
\end{proof}
%$\bar{G}_{11}(\theta), \bar{G}_{12}(\theta),\bar{G}_{21}(\theta) ,\bar{G}_{22}(\theta)$ are all defined analogously. 

\textbf{Checking Assumptions \ref{assump:sa1}, \ref{assump:sa2} and \ref{assump:sa3}.} Clearly Assumption~\ref{assump:sa1}, Assumption~\ref{assump:sa2} and Assumption~\ref{assump:sa3}(a) are satisfied. To check Assumption~\ref{assump:sa3}(b) and (c), we have the following Lemma. 

\begin{lemma}\label{lem:critic:Gh_bounded}
(a) For any $z,z'\in \mathcal{Z}$, we have, 
\[\Vert \tilde{G}_i(z,z')\Vert_\infty \leq  2 + \frac{1}{\rescale
}:= G_{\max}, \quad \Vert h_i(z)\Vert_\infty \leq \max(\rescale ,1) \bar{r}:= h_{\max}.\]
As a result, $\Vert G_i^\theta(z)\Vert_\infty \leq G_{\max}$ and $\Vert \bar{G}_i(\theta)\Vert_\infty \leq G_{\max}$, $\Vert \bar{h}_i(\theta) \Vert_\infty \leq  h_{\max}$.

(b) We also have that, 
\begin{align*}
   \Big\Vert \E_\theta  \big[G_i^{\theta}(z(t)) - \bar{G}_i(\theta)   |z(0)=z\big] \Big\Vert_\infty &\leq  2 G_{\max}  c_\infty\mu_D^t,\\
   \Big\Vert \E_\theta  \big[h_i(z(t)) - \bar{h}_i(\theta)   |z(0)=z\big]\Big\Vert_\infty &\leq  2 h_{\max}  c_\infty\mu_D^t.
\end{align*}
As a consequence, for any $z$, $\Vert \hat{G}_i^\theta(z)\Vert_\infty \leq 2 G_{\max}  c_\infty \frac{1}{1-\mu_D} $, $\Vert \hat{h}_i^\theta(z)\Vert_\infty \leq 2 h_{\max}  c_\infty \frac{1}{1-\mu_D} $.
\end{lemma}
\begin{proof} Part (a) follows directly from the definition as well as the bounded reward (Assumption~\ref{assump:reward}). Part (b) is a consequence of Lemma~\ref{lem:mixing}.  In details, given $z$, let $d^t$ be the distribution of $z(t)$ starting form $z$. Then, 
\begin{align*}
    \Big\Vert \E_\theta  \big[G_i^{\theta}(z(t)) - \bar{G}_i(\theta)   |z(0)=z\big] \Big\Vert_\infty &= \Vert \E_{z\sim d^t} G_i^\theta (z) - \E_{z\sim \sd^\theta } G_i^\theta(z)\Vert_\infty\\
    &= \Vert \sum_{z} (d^t(z) - \sd^\theta(z)) G_i^\theta(z)\Vert_\infty\\
    &\leq \sum_z | d^t(z) - \sd^\theta(z)| \Vert G_i^\theta(z)\Vert_\infty\\
    &\leq   G_{\max} \Vert d^t - \pi^\theta\Vert_1 \\
    &=G_{\max} \Vert ((P^\theta)^\top )^t (d^0 - \pi^\theta)\Vert_1 \leq G_{\max} 2 c_\infty\mu_D^t.
    \end{align*}
The proof for $h_i$ is similar. 
\end{proof}

Next, the following Lemma~\ref{lem:critic:lipschitz} shows the Lipschitz condition in Assumption~\ref{assump:sa3} (d) and (e) are true. The proof of Lemma~\ref{lem:critic:lipschitz} is postponed to \Cref{subsec:critic:lipschitz}

\begin{lemma}\label{lem:critic:lipschitz} The following holds.
\begin{itemize}
    \item[(a)] $P^\theta$ and $\sd^\theta$ are Lipschitz in $\theta$. 
    \item[(b)]   $\bar{G}_i(\theta)$ and $\bar{h}_i(\theta)$ are Lipschitz in $\theta$. 
    \item[(c)] For any $z$, $ [P^\theta\hat{h}_i^{\theta}](z) $ and $[P^\theta \hat{G}_i^\theta](z)$ are Lipschitz in $\theta$ with the Lipschitz constant independent of $z$. 
\end{itemize}
\end{lemma}

\textbf{Checking Assumption~\ref{assump:sa4}.} Recall that $\theta(t+1) = \theta(t) + \eta_t \Gamma(\mathbf{\hat{Q}}^t) \hat{g}(t)$. Note that $\Vert \hat{g}_i(t)\Vert \leq L_i \max_{j}\Vert\hat{Q}_j^t\Vert_\infty$. By the definition of $\Gamma(\mathbf{\hat{Q}}^t)$, we have almost surely $\Vert \Gamma(\mathbf{\hat{Q}}^t) \hat{g}_i(t)\Vert \leq L_i$ for all $t$. As such, almost surely, for all $t$, $\Vert \Gamma(\mathbf{\hat{Q}}^t) \hat{g}(t)\Vert \leq L$. This, together with our selection of $\eta_t$ (Assumption~\ref{assump:stepsize}), shows that Assumption~\ref{assump:sa4} is satisfied.

\textbf{Checking Assumption~\ref{assump:sa5}.} Recall that $\xi_i^{t+1} = G_i^{\theta(t)}(z(t)) - \tilde{G}_i(z(t),z(t+1)) $. We have clearly $\E \xi_i^{t+1}|\mathcal{F}_t = 0$ per the definition of $G_i^\theta(z)$. Further, $\Vert \xi_i^{t+1}\Vert_\infty \leq 2 G_{\max}.$ So Assumption~\ref{assump:sa5} is satisfied. 

\textbf{Checking Assumption~\ref{assump:sa6}.} Finally, we check Assumption~\ref{assump:sa6}, the assumption that $\bar{G}_i(\theta)$ is uniformly positive definite. This is done in the following Lemma~\ref{lem:critic:positive_definite}, whose proof is postponed to \Cref{subsec:critic:positive_definite}.

\begin{lemma}\label{lem:critic:positive_definite}
We have when $\rescale = \frac{1}{\sigma \sqrt{(1-\mu_D)} }$, then for any $\theta$, $x_i^\top \bar{G}_i(\theta) x_i \geq \frac{1}{2}(1-\mu_D) \sigma^2  \Vert x_i\Vert^2.$
\end{lemma}

Given $\theta$,  let $x_i^\theta = [\rescale\hat{\mu}_i^{\theta};\hat{Q}_i^{\theta}]$ be the unique solution to $ \bar{h}_i(\theta) - \bar{G}_i(\theta) x_i  = 0$. 
Now that Assumptions~\ref{assump:sa1} to \ref{assump:sa6} are satisfied, by Theorem~\ref{thm:SA} we immediately have almost surely $\lim_{t\rightarrow\infty}   \Vert  x_i^t- [\bar{G}_i(\theta(t))]^{-1}\bar{h}_i(\theta(t)) \Vert =0$, and $\sup_{t\geq 0}\Vert x_i^t\Vert <\infty$. As $x_i^t = [\rescale \hat{\mu}_i^t; \hat{Q}_i^t]$, this directly implies $\lim_{t\rightarrow\infty}   \hat{Q}_i^t - \hat{Q}_i^{\theta(t)} =0$ and $\sup_{t\geq 0} \Vert \hat{Q}_i^t\Vert_\infty \leq \infty$. This proves part (b) of Theorem~\ref{thm:critic}. For part (a), we show the following Lemma~\ref{lem:critic:fixed_point} on the property of $x_i^\theta$, whose proof is postponed to \Cref{subsec:critic:fixed_point}. With Lemma~\ref{lem:critic:fixed_point}, the proof of Theorem~\ref{thm:critic} is concluded.

\begin{lemma}\label{lem:critic:fixed_point}
Given $\theta$,  the solution $x_i^\theta = [\rescale\hat{\mu}_i^{\theta};\hat{Q}_i^{\theta}]$ to $ \bar{h}_i(\theta) - \bar{G}_i(\theta) x_i  = 0$ satisfies $\hat{\mu}_i^\theta = J_i(\theta)$. Further, there exists some $c_i^\theta\in\R$ s.t. 
\begin{align}
  \Vert \Phi_i \hat{Q}_i^\theta + c_i^\theta\mathbf{1} - Q_i^\theta \Vert_{\SD^\theta}\leq \frac{  c \rhok}{1-\mu_{D}},
\end{align}
where $\mathbf{1}$ is the all one vector in $\R^{\mathcal{Z}}$.
\end{lemma}

\subsection{Proof of Lemma~\ref{lem:critic:lipschitz}}\label{subsec:critic:lipschitz}
 To show (a), notice that, $P^\theta(s',a'|s,a) = P(s'|s,a) \zeta^\theta(a'|s')$. Therefore, 
 \begin{align*}
     \Vert P^\theta - P^{\bar\theta}\Vert_\infty &= \max_{s,a} \sum_{s',a'} P(s'|s,a) |\zeta^\theta(a'|s') - \zeta^{\bar\theta}(a'|s')| \\
     &\leq L\Vert\theta - \bar{\theta}\Vert \max_{s,a} \sum_{s',a'} P(s'|s,a)\\
     &= L |\mathcal{A}| \Vert\theta - \bar{\theta}\Vert := L_P \Vert\theta - \bar{\theta}\Vert,
 \end{align*}
    where in the inequality, we have used that for any $a\in\mathcal{A}$, $s\in\mathcal{S}$, as $\Vert \nabla_{\theta_i} \log \zeta^{\theta}(a|s) \Vert=\Vert \nabla_{\theta_i} \log \zeta_i^{\theta_i}(a_i|s_i) \Vert \leq L_i$ (Assumption~\ref{assump:gradient}), we have $\Vert \nabla_\theta \zeta^\theta(a|s) \Vert \leq \Vert \nabla_\theta \log \zeta^\theta(a|s) \Vert \leq \sqrt{\sum_{i\in\mathcal{N}} L_i^2} = L $.
    
    Next, we show $\sd^\theta $ is Lipschitz continuous in $\theta$. Notice that $\sd^\theta$ satisfies $\sd^\theta = ( P^\theta)^\top \sd^\theta$. As such, we have,
    \begin{align*}
        \sd^\theta - \sd^{\bar\theta} &= (P^\theta)^\top(\pi^\theta - \sd^{\bar{\theta}}) + (P^\theta - P^{\bar{\theta}})^\top \sd^{\bar\theta}\\
        &=((P^\theta)^\top)^k(\pi^\theta - \sd^{\bar{\theta}}) + \sum_{\ell=0}^{k-1} ((P^\theta)^\top)^{\ell} (P^\theta - P^{\bar{\theta}})^\top \sd^{\bar\theta}=  \sum_{\ell=0}^{\infty} ((P^\theta)^\top)^{\ell} (P^\theta - P^{\bar{\theta}})^\top \sd^{\bar\theta}.
    \end{align*}
    Notice that by Lemma~\ref{lem:mixing},
    \begin{align*}
        \Vert ((P^\theta)^\top)^{\ell} (P^\theta - P^{\bar{\theta}})^\top \sd^{\bar\theta}\Vert_1 \leq c_\infty \mu_D^\ell \Vert (P^\theta - P^{\bar{\theta}})^\top \sd^{\bar\theta}\Vert_1\leq c_\infty\mu_D^\ell \Vert P^\theta - P^{\bar\theta}\Vert_\infty.
    \end{align*}
    Therefore, we have
    \begin{align*}
        \Vert \sd^\theta - \sd^{\bar{\theta}} \Vert_1 &\leq  \frac{c_\infty}{1-\mu_D} \Vert P^\theta - P^{\bar\theta}\Vert_\infty \leq  \frac{c_\infty}{1-\mu_D} L_P\Vert \theta - \bar{\theta}\Vert.
    \end{align*}
 So we are done for part (a). 
    
    For part (b), notice that $\bar{h}_i(\theta)$ depends on $\theta$ only through $\sd^\theta$ and is linear in $\sd^\theta$. As a result $\bar{h}_i(\theta)$ is Lipschitz in $\theta$. For similar reasons, for $\bar{G}_i(\theta)$ we only need to show $\SD^\theta P^\theta$ is Lipschitz in $\theta$. This is true because both $\SD^\theta$ and $P^\theta$ are Lipschitz in $\theta$, and they themselves are bounded. 
    
    \smallskip
    For part (c), fixing any initial $z$, let $d^{\theta,t}$ be the distribution of $z(t)$ under policy $\theta$. We first show that $d^{\theta,t} - \pi^\theta$ is Lipschitz in $\theta$ with Lipschitz constant geometrically decaying in $t$. To this end, note that
\begin{align*}
 d^{\theta,t} - \pi^\theta - (d^{\bar\theta,t} - \pi^{\bar\theta})  
    &= (P^\theta)^\top ( d^{\theta,t-1} - \pi^\theta) - (P^{\bar{\theta}})^\top (d^{\bar\theta,t-1} - \pi^{\bar\theta})\\
    &= (P^\theta)^\top [ d^{\theta,t-1} - \pi^\theta- (d^{\bar\theta,t-1} - \pi^{\bar\theta} )] + (P^\theta - P^{\bar{\theta}})^\top (d^{\bar\theta,t-1} - \pi^{\bar\theta})\\
    &= ((P^\theta)^t )^\top (\pi^{\bar\theta}-\pi^\theta) +\sum_{\ell=0}^{t-1}((P^\theta)^\ell )^\top (P^\theta - P^{\bar{\theta}})^\top (d^{\bar\theta,t-\ell-1} - \pi^{\bar\theta}).
\end{align*}
As such, we have,
\begin{align}
    \Vert d^{\theta,t} - \pi^\theta - (d^{\bar\theta,t} - \pi^{\bar\theta})  \Vert_1 &\leq c_\infty \mu_D^t \Vert \pi^{\bar\theta}-\pi^\theta\Vert_1 + \sum_{\ell=0}^{t-1} c_\infty\mu_D^{\ell} \Vert P^\theta - P^{\bar\theta}\Vert_\infty \Vert d^{\bar\theta,t-\ell-1} - \pi^{\bar\theta}\Vert_1 \nonumber \\ 
    &\leq c_\infty \mu_D^t \frac{c_\infty}{1-\mu_D} L_P\Vert \theta-\bar{\theta}\Vert + \sum_{\ell=0}^{t-1} c_\infty\mu_D^\ell L_P \Vert \theta - \bar\theta\Vert 2 c_\infty \mu_D^{t-\ell-1} \nonumber \\
    &= \frac{c_\infty^2 L_P}{1-\mu_D} \mu_D^t  \Vert \theta-\bar{\theta}\Vert + 2 c_\infty^2 L_P  t \mu_D^{t-1}  \Vert \theta - \bar\theta\Vert \nonumber\\
   & < \frac{5c_\infty^2 L_P}{1-\mu_D } (\frac{1+\mu_D}{2})^t \Vert \theta - \bar{\theta}\Vert.\label{eq:critic:d_lipschitz}
\end{align}
Next, we turn to $\hat{G}_i^\theta(z)$ and show its Lipschitz continuity in $\theta$. Note that by definition, 
\begin{align*}
     \hat{G}_i^\theta(z) &= \E_\theta [\sum_{t=0}^\infty [G_i^{\theta}(z(t)) - \bar{G}_i(\theta)  ] |z(0)=z] = \sum_{t=0}^\infty [ \E_{z'\sim d^{\theta,t}} G_i^\theta(z')  -\E_{z'\sim \pi^\theta}G_i^\theta(z')]\\
     &=\sum_{t=0}^\infty \sum_{z'\in\mathcal{Z}} (d^{\theta,t}(z') - \pi^\theta(z')) G_i^\theta(z').
\end{align*}

As such,
\begin{align*}
    &\Vert \hat{G}_i^\theta(z) - \hat{G}_i^{\bar\theta}(z)\Vert_\infty \\
    &\leq \sum_{t=0}^\infty \sum_{z'\in\mathcal{Z}} \Big\Vert (d^{\theta,t}(z') - \pi^\theta(z')) G_i^\theta(z') - (d^{\bar\theta,t}(z') - \pi^{\bar\theta}(z')) G_i^{\bar\theta}(z') \Big\Vert_\infty\\
    &\leq  \sum_{t=0}^\infty \sum_{z'\in\mathcal{Z}} \Big[ |d^{\theta,t}(z') - \pi^\theta(z') - (d^{\bar\theta,t}(z') - \pi^{\bar\theta}(z')) | \Vert G_i^\theta(z')\Vert_\infty + |d^{\bar\theta,t}(z') - \pi^{\bar\theta}(z')| \Vert G_i^{\theta}(z')- G_i^{\bar\theta}(z')\Vert_\infty \Big]\\
    &\leq \sum_{t=0}^\infty \Big[\Vert d^{\theta,t} - \pi^\theta - (d^{\bar\theta,t} - \pi^{\bar\theta})\Vert_1 G_{\max} + \Vert d^{\bar\theta,t} - \pi^{\bar\theta}\Vert_1 \sup_{z'}\Vert  G_i^{\theta}(z')- G_i^{\bar\theta}(z') \Vert_\infty  \Big]\\
    &\leq \sum_{t=0}^\infty \Big[  \frac{5c_\infty^2 L_P G_{\max}}{1-\mu_D } (\frac{1+\mu_D}{2})^t \Vert \theta - \bar{\theta}\Vert  + 2 c_\infty \mu_D^t \sup_{z'}\Vert  G_i^{\theta}(z')- G_i^{\bar\theta}(z') \Vert_\infty \Big].
\end{align*}
 Since $G_i^\theta(z')$ depends on $\theta$ only through $P^\theta$ and is linear in $P^\theta$, $G_i^\theta(z')$ is Lipschitz in $\theta$. Therfore, in the above summation, each summand can be written as some geometrically decaying term times $\Vert\theta - \bar{\theta}\Vert$. As such, $ \hat{G}_i^\theta(z)$ is Lipschitz in $\theta$, and the Lipschitz constant can be made independent of $z$ by taking the sup over the finite set $z\in\mathcal{Z}$. As a result, $ [P^\theta \hat{G}_i^\theta] (z) = \sum_{z'} P^\theta(z'|z) \hat{G}_i^\theta(z') $ is Lipschitz in $\theta$ as well since both $P^\theta$ and $\hat{G}_i^\theta(z')$ are Lipschitz in $\theta$ and bounded. 

The proof for the Lipschitz continuity of $P^\theta \hat{h}_i^\theta(z)$ is similar and is hence omitted. Therefore, part (c) is done and the proof is concluded.
\subsection{Proof of Lemma~\ref{lem:critic:positive_definite}} \label{subsec:critic:positive_definite}
Recall that,
\begin{align*}
     \bar{G}_i(\theta) =   \left[ \begin{array}{cc}
 1 & 0\\
\frac{1}{\rescale} \Phi_i^\top \sd^\theta  & \Phi_i^\top \SD^\theta \big[   \Phi_i - P^\theta\Phi_i \big] 
    \end{array}\right].
\end{align*}
Let $x_i = [\hat{\mu}_i, \hat{Q}_i]$ and define $\hat{\Phi}_i = \Phi_i - \mathbf{1} (\sd^\theta)^\top \Phi_i$. Then, 
\begin{align*}
    \Phi_i^\top D^\theta \Phi_i & = \hat\Phi_i^\top D^\theta \hat\Phi_i + \Phi_i^\top \sd^\theta \mathbf{1}^\top D^\theta \hat\Phi_i + \hat\Phi_i^\top D^\theta \mathbf{1} (\sd^\theta)^\top \Phi_i + \Phi_i^\top \sd^\theta \mathbf{1}^\top D^\theta \mathbf{1} (\sd^\theta)^\top \Phi_i\\
    &= \hat\Phi_i^\top D^\theta \hat\Phi_i + \Phi_i^\top \sd^\theta   (\sd^\theta)^\top \Phi_i,\\
    \Phi_i^\top D^\theta P^\theta \Phi_i &= \hat\Phi_i^\top D^\theta P^\theta \hat\Phi_i + \Phi_i^\top \sd^\theta \mathbf{1}^\top D^\theta P^\theta \hat\Phi_i + \hat\Phi_i^\top D^\theta P^\theta \mathbf{1} (\sd^\theta)^\top \Phi_i + \Phi_i^\top \sd^\theta \mathbf{1}^\top D^\theta P^\theta \mathbf{1} (\sd^\theta)^\top \Phi_i\\
    &= \hat\Phi_i^\top D^\theta P^\theta \hat\Phi_i + \Phi_i^\top \sd^\theta   (\sd^\theta)^\top \Phi_i.
\end{align*}
As such, 
$$\Phi_i^\top D^\theta \Phi_i - \Phi_i^\top D^\theta P^\theta \Phi_i = \hat\Phi_i^\top D^\theta \hat\Phi_i - \hat\Phi_i^\top D^\theta P^\theta \hat\Phi_i , $$
from which, we have using Assumption~\ref{assump:ergodicity},
\begin{align}
    \hat{Q}_i^\top (\Phi_i^\top D^\theta \Phi_i - \Phi_i^\top D^\theta P^\theta \Phi_i) \hat{Q}_i &\geq \Vert \hat{\Phi}_i \hat{Q}_i\Vert_{\SD^\theta}^2 - \Vert \hat{\Phi}_i \hat{Q}_i\Vert_{\SD^\theta} \Vert P^\theta \hat{\Phi}_i \hat{Q}_i\Vert_{\SD^\theta}  \nonumber \\
    &= \Vert \hat{\Phi}_i \hat{Q}_i\Vert_{\SD^\theta}^2 - \Vert \hat{\Phi}_i \hat{Q}_i\Vert_{\SD^\theta} \Vert ( P^\theta - \mathbf{1}(\sd^\theta)^\top)\hat{\Phi}_i \hat{Q}_i\Vert_{\SD^\theta} \nonumber \\
    &\geq  \Vert \hat{\Phi}_i \hat{Q}_i\Vert_{\SD^\theta}^2 - \mu_D \Vert \hat{\Phi}_i \hat{Q}_i\Vert_{\SD^\theta}^2 \nonumber\\
    &= (1 - \mu_D) \hat{Q}_i^\top \hat{\Phi}_i^\top \SD^\theta \hat{\Phi}_i \hat{Q}_i \nonumber\\
    &= (1 - \mu_D) \hat{Q}_i^\top ( {\Phi}_i^\top \SD^\theta {\Phi}_i -  \Phi_i^\top \sd^\theta (\sd^\theta)^\top \Phi_i )\hat{Q}_i\nonumber\\
    &\geq (1-\mu_D)\sigma^2 \Vert \hat{Q}_i\Vert^2,\label{eq:critic:uniform_pd_eq}
   % &= (1 - \mu_D) \Big[\sum_{z_\nik\in\mathcal{Z}_\nik/\{\tilde{z}_\nik\}} \sd^\theta(z_{\nik})\hat{Q}_i(z_{\nik})^2 - (\sum_{z_\nik\in\mathcal{Z}_\nik/\{\tilde{z}_\nik\}}\sd^\theta(z_{\nik})\hat{Q}_i(z_{\nik}) )^2\Big]
\end{align}
where the last step is due to the following. Let $v\in\R^{\hat{\mathcal{Z}}_{\nik}}$ be the marginalized distribution of $z_{\nik}\in \hat{\mathcal{Z}}_{\nik}$ under $\sd^\theta$, i.e. $v(z_\nik) = \pi^\theta (z_\nik)$. Using $v(z_{\nik}) \geq \sigma$ and $\sum_{z_{\nik}\in\hat{\mathcal{Z}}_{\nik}} v(z_{\nik}) \leq 1 - \sigma$ (Assumption~\ref{assump:ergodicity}), we have,
\begin{align*}
    {\Phi}_i^\top \SD^\theta {\Phi}_i - \Phi_i^\top \pi^\theta (\pi^\theta)^\top \Phi_i &= \diag(v) - v v^\top = \diag(v)^{\frac{1}{2}} (I - \diag(v)^{-\frac{1}{2}} v ( \diag(v)^{-\frac{1}{2}} v )^\top ) \diag(v)^{\frac{1}{2}} \\
    & \succeq (1 - \Vert \diag(v)^{-\frac{1}{2}} v\Vert^2) \diag(v) \\
    &\succeq \sigma^2 I. 
\end{align*}
Building on \eqref{eq:critic:uniform_pd_eq}, the rest of the proof follows easily. We have, 
\begin{align*}
    x_i^\top \bar{G}_i (\theta) x_i &\geq \hat{\mu}_i^2 + (1-\mu_D) \sigma^2 \Vert\hat{Q}_i\Vert^2 +  \frac{1}{\rescale}\hat{Q}_i^\top \Phi_i^\top \sd^\theta \hat{\mu}_i  \\
    &\geq \hat{\mu}_i^2 + (1-\mu_D) \sigma^2 \Vert\hat{Q}_i\Vert^2 -  \frac{1}{\rescale} \Vert \hat{Q}_i \Vert | \hat{\mu}_i|\\
    &\geq \min(\frac{1}{2}, \frac{1}{2}(1-\mu_D) \sigma^2 )\Vert x_i\Vert^2,
\end{align*}
where we have used 
\[\frac{1}{2} \hat{\mu}_i^2 + \frac{1}{2} (1-\mu_D) \sigma^2 \Vert\hat{Q}_i\Vert^2 \geq \sigma \sqrt{(1-\mu_D)} \Vert \hat{Q}_i\Vert |\hat{\mu}_i| \geq \frac{1}{c'} \Vert \hat{Q}_i\Vert |\hat{\mu}_i| . \]

\subsection{Proof of Lemma~\ref{lem:critic:fixed_point}} \label{subsec:critic:fixed_point}
By the definition of $\bar{G}_i(\theta)$ and $\bar{h}_i(\theta)$, we have $\hat{\mu}_i^{\theta} = (\pi^\theta)^\top r_i = J_i(\theta) $, the average reward at node $i$ under policy $\theta$, and $\hat{Q}_i^\theta \in\R^{\hat{\mathcal{Z}}_{\nik}}$ is the solution to the following linear equation (the solution must be unique due to Lemma~\ref{lem:critic:positive_definite}),
\begin{align}\label{eq:critic:fixed_point_q}
     0 &=-\Phi_i^\top \SD^\theta \hat{\mu}_i^\theta \mathbf{1} + \Phi_i^\top \SD^\theta \big[ P^\theta\Phi_i - \Phi_i \big] \hat{Q}_i^\theta+ \Phi_i^\top\SD^\theta r_i \nonumber \\
        &=\Phi_i^\top \SD^\theta [r_i -\hat{\mu}_i^\theta\mathbf{1} +  P^\theta\Phi_i \hat{Q}_i^\theta] - \Phi_i^\top\SD^\theta\Phi_i\hat{Q}_i^\theta \nonumber\\
                &=\Phi_i^\top \SD^\theta [r_i - J_i(\theta) \mathbf{1} +  P^\theta\Phi_i \hat{Q}_i^\theta] - \Phi_i^\top\SD^\theta\Phi_i\hat{Q}_i^\theta  .
                %&= \Lambda_i^\theta \Pi_i^\theta \td(\Phi_i\hat{Q}_i) - \Lambda_i^\theta \hat{Q}_i 
        %&= \Phi_i^\top \diag(d) \td(\Phi_i\hat{Q}_i) -  \Phi_i^\top\diag(d)\Phi_i\hat{Q}_i\\
%&= -D\hat{Q}_i + D \Pi^d\td(\Phi_i\hat{Q}_i)\\
%&=  -D\hat{Q}_i +D  g^d(\hat{Q}_i),
\end{align}
To understand the solution of \eqref{eq:critic:fixed_point_q}, we define an equivalent expanded equation, whose solution can be related to the Bellman operator. For this purpose, define $\tilde{\Phi}_i\in\R^{\mathcal{Z}\times {\mathcal{Z}}_{\nik}}$ to be a matrix with each row indexed by $z\in \mathcal{Z}$ and each column indexed by $z_{\nik}' \in  \mathcal{Z}_{\nik}$ and $\tilde{\Phi}_i(z,z_{\nik}') = 1$ if $z_{\nik} = z_{\nik}'$ and $0$ elsewhere. In other words, $\tilde{\Phi}_i$ is essentially $\Phi_i$ with the additional column corresponding to the dummy state-action pair $\tilde{z}_{\nik}$. Consider the following equations on $\bar{Q}_i^\theta\in\R^{\mathcal{Z}_{\nik}}$
\begin{subequations}\label{eq:critic:fixed_point_q_equi}
\begin{align}
0&= \tilde{\Phi}_i^\top \SD^\theta [r_i - J_i(\theta) \mathbf{1} +  P^\theta \tilde\Phi_i \bar{Q}_i^\theta] - \tilde\Phi_i^\top\SD^\theta\tilde\Phi_i\bar{Q}_i^\theta, \label{eq:critic:fixed_point_q_equi_1}  \\
0 &=  \bar{Q}_i^\theta(\tilde{z}_\nik) .\label{eq:critic:fixed_point_q_equi_2}
\end{align}
\end{subequations}
\textbf{Claim 1:} The equations \eqref{eq:critic:fixed_point_q} and \eqref{eq:critic:fixed_point_q_equi} are equivalent in the sense that both have unique solutions, and the solutions are related by $\hat{Q}_i^\theta(z_\nik) = \bar{Q}_i^\theta(z_{\nik}),\forall z_{\nik} \in \hat{\mathcal{Z}}_\nik$. 

Before we prove the claim, we first show \eqref{eq:critic:fixed_point_q_equi_1} can be actually reformulated as the fixed point equation related to the Bellman operator.

\textbf{Reformulation of \eqref{eq:critic:fixed_point_q_equi_1} as fixed point equation}.    It is easy to check that $\tilde{\SD}^\theta_i = \tilde{\Phi}_i^\top \SD^\theta \tilde{\Phi}_i \in \mathbb{R}^{{\mathcal{Z}}_{\nik}\times {\mathcal{Z}}_{\nik}}$ is a diagonal matrix, and the $z_{\nik}$'th diagonal entry is the marginal probability of $z_{\nik}$ under $\sd^\theta$, which is non-zero by Assumption~\ref{assump:ergodicity}. Therefore, $\tilde{\Phi}_i^\top\SD^\theta \tilde{\Phi}_i  $ is invertable and matrix $\Pi_i^\theta = (\tilde{\Phi}_i^\top\SD^\theta \tilde{\Phi}_i)^{-1}\tilde{\Phi}^\top_i \SD^\theta$ is well defined. Further, the $z_{\nik}$'th row of $\Pi_i^\theta$ is in fact the conditional distribution of the full state $z$ given $z_{\nik}$. So, $\Pi_i^\theta$ must be a stochastic matrix and is non-expansive in infinity norm.  Let $\td_i^\theta(Q_i) = r_i - J_i(\theta) \mathbf{1} + P^\theta Q_i$ be the Bellman operator for reward $r_i$. 
    Further, define operator $g:\R^{\mathcal{Z}_\nik}\rightarrow \R^{\mathcal{Z}_\nik}$ given by $g( \tilde{Q}_i)=\Pi_i^\theta \td_i^\theta \tilde{\Phi}_i \tilde{Q}_i$ for $\tilde{Q}_i\in \R^{\mathcal{Z}_\nik}$. Then, \eqref{eq:critic:fixed_point_q_equi_1} is equvalent to the fixed point equation of operator $g$, $g(\bar{Q}_i^\theta) = \bar{Q}_i^\theta$. Our next claim studies the structure of the fixed points of $g$.
    
\textbf{Claim 2:} Define $\subs = \{Q_i\in\R^{\mathcal{Z}}: \E_{z\sim\sd^\theta }  Q_i(z) = 0 \}$, and $\ssubs = \{\tilde{Q}_i\in\R^{\mathcal{Z}_\nik}: \E_{z\sim\sd^\theta} \tilde{Q}_i(z_\nik) = 0  \} $. We claim that $g$ has a unique fixed point within $\ssubs$ which we denote as $\tilde{Q}_i^\theta$. Further, all fixed points of $g$ are the set $\{\tilde{Q}_i^\theta + c_i\mathbf{1}:c_i\in\R\}$.

\textbf{Proof of Claim 2.} We in fact show $g$ maps $\ssubs$ to $\ssubs$ and is a contraction in $\Vert\cdot\Vert_{\tilde{\SD}_i^\theta}$ norm when restricted to $\ssubs$, which will guarantee the existence and uniqueness of $\tilde{Q}_i^\theta$. To see this, we check the following steps.
    \begin{itemize}
        \item  $\tilde{\Phi}_i$ maps $\ssubs$ to $\subs$ and preserves metric from $\Vert\cdot\Vert_{\tilde{\SD}_i^\theta}$ to $\Vert\cdot\Vert_{\SD^\theta}$. To see this, note that $ \tilde{\Phi}_i^\top D^{\theta} \tilde{\Phi}_i   = \tilde{\SD}_i^\theta$. 
        \item $\td_i^\theta$ maps $\subs$ to $\subs$ and further, it is a $\mu_D$ contraction in $\Vert \cdot\Vert_{\SD^\theta}$ when restricted to $\subs$. To see this, note that for $Q_i,Q_i'\in\subs$, $\Vert \td_i^\theta(Q_i) - \td_i^\theta(Q_i') \Vert_{\SD^\theta} =\Vert P^\theta(Q_i - Q_i')\Vert_{\SD^\theta} =  \Vert (P^\theta- \mathbf{1}(\sd^\theta)^\top )(Q_i - Q_i')\Vert_{\SD^\theta}\leq \mu_D \Vert Q_i - Q_i'\Vert_{\SD^\theta}$, where we have used Assumption~\ref{assump:ergodicity}.  
        \item $\Pi_i^\theta$ maps $\subs$ to $\ssubs$ and is non-expensive from $\Vert \cdot\Vert_{D^\theta}$ to $\Vert\cdot\Vert_{\tilde{D}_i^\theta}$. To see this, notice that $(\Pi_i^\theta Q_i) (z_{\nik}) = \E_{z'\sim \pi^\theta(z'|z'_\nik = z_\nik )} Q_i(z')$. As such, when $Q_i\in\subs$, $\E_{z\sim\sd^\theta} (\Pi_i^\theta Q_i)(z_\nik) = \E_{z'\sim \pi^\theta} Q_i(z')=0$, which shows $\Pi_i^\theta Q_i\in\ssubs$. 
Finally, one can check $\Pi_i^\theta$ is non-expensive from $\Vert \cdot\Vert_{D^\theta}$ to $\Vert\cdot\Vert_{\tilde{D}_i^\theta}$ by noting $(\tilde{D}_i^\theta)^{1/2}  \Pi_i^\theta (\SD^\theta)^{-1/2}  = (\tilde{D}_i^\theta)^{-1/2} \tilde{\Phi}_i^\top  (\SD^\theta)^{1/2}$, the rows of which are orthornormal vectors.  
    \end{itemize}

Combining these relations, $g$ maps $\ssubs$ to itself and further, we have for $\tilde{Q}_i,\tilde{Q}_i'\in\ssubs$,
\begin{align*}
    \Vert g(\tilde{Q}_i) - g(\tilde{Q}_i')\Vert_{\tilde\SD_i^\theta}& =  \Vert \Pi_i^\theta (\td_i^\theta \tilde{\Phi}_i \tilde{Q}_i - \td_i^\theta\tilde{\Phi}_i \tilde{Q}_i' )\Vert_{\tilde\SD_i^\theta}\leq \Vert \td_i^\theta (\tilde{\Phi}_i \tilde{Q}_i )- \td_i^\theta( \tilde{\Phi}_i \tilde{Q}_i' )\Vert_{\SD^\theta}\\
    &\leq \mu_D \Vert\tilde{\Phi}_i (\tilde{Q}_i - \tilde{Q}_i')\Vert_{\SD^\theta} =\mu_D \Vert \tilde{Q}_i - \tilde{Q}_i'\Vert_{\tilde\SD_i^\theta},
\end{align*}
which shows $g$ is a contraction when restricted to $\ssubs$. This shows $g$ has a unique fixed point within $\ssubs$, which we denote by $\tilde{Q}_i^\theta$. Further, note for any $c_i\in\R$,
\begin{align*}
    g(\tilde{Q}_i + c_i\mathbf{1}) &=\Pi_i^\theta \td_i^\theta \tilde{\Phi}_i (\tilde{Q}_i + c_i\mathbf{1}) =\Pi_i^\theta \td_i^\theta ( \tilde{\Phi}_i \tilde{Q}_i + c_i\mathbf{1})\\
    & = \Pi_i^\theta [\td_i^\theta \tilde{\Phi}_i \tilde{Q}_i + c_i\mathbf{1} ] = g(\tilde{Q}_i) + c_i\mathbf{1}.
\end{align*}
Therefore, let $\tilde{Q}_i$ be a fixed point of $g$, then $\tilde{Q}_i - \mathbf{1}\E_{z\sim\sd^\theta} \tilde{Q}_i(z_{\nik})$ will be a fixed point of $g$ within $\ssubs$. As such, the set of fixed point of $g$ can be written in the form $\{\tilde{Q}_i^\theta + c_i\mathbf{1}: c_i\in\R\}$. \qedsymbol

We are now ready to prove Claim 1. 

\textbf{Proof of Claim 1.} By Claim 2, the set $\{\tilde{Q}_i^\theta + c_i\mathbf{1}: c_i\in\R\}$ characterizes the solution to equation \eqref{eq:critic:fixed_point_q_equi_1}. 
Therefore, \eqref{eq:critic:fixed_point_q_equi} must have a unique solution $\bar{Q}_i^\theta = \tilde{Q}_i^\theta - \tilde{Q}_i^\theta(\tilde{z}_\nik)\mathbf{1}$. 
Since \eqref{eq:critic:fixed_point_q_equi_1} is a overdetermined equation, we can essentially remove one row corresponding to $\tilde{z}_\nik$, and then plug in $\bar{Q}_i^\theta(\tilde{z}_\nik)=0$. 
This corresponds exactly to the equation in \eqref{eq:critic:fixed_point_q}. 
As such, the solution of \eqref{eq:critic:fixed_point_q} is the solution of \eqref{eq:critic:fixed_point_q_equi}, removing the entry in $\tilde{z}_\nik$. \qedsymbol

By Claim 1 and Claim 2, we have 
\begin{align*}
    \Phi_i\hat{Q}_i^\theta = \tilde{\Phi}_i\bar{Q}_i^\theta = \tilde{\Phi}_i \tilde{Q}_i^\theta - \tilde{Q}_i^\theta(\tilde{z}_\nik) \mathbf{1}.
\end{align*}
As such, we can set $c_i^\theta = \tilde{Q}_i^\theta(\tilde{z}_\nik)$ and get,
\begin{align}
 \Vert \Phi_i\hat{Q}_i^\theta + c_i^\theta\mathbf{1}- Q_i^\theta \Vert_{\SD^\theta} = \Vert \tilde{\Phi}_i\tilde{Q}_i^\theta - Q_i^\theta \Vert_{\SD^\theta}. \label{eq:critic:fixed_point_error_1}
\end{align}
Finally, we bound $\Vert \tilde{\Phi}_i\tilde{Q}_i^\theta - Q_i^\theta \Vert_{\SD^\theta}$. We have, using $\tilde{Q}_i^\theta$ is a fixed point of $g(\cdot)$ and $Q_i^\theta$ is a fixed point of $\td_i^\theta$,
\begin{align*}
    \tilde{\Phi}_i \tilde{Q}_i^{\theta} - Q_i^\theta& = \tilde{\Phi}_i \tilde{Q}_i^{\theta} - \tilde{\Phi}_i \Pi_i^\theta Q_i^\theta + \tilde{\Phi}_i \Pi_i^\theta Q_i^\theta - Q_i^\theta\\
    &=  \tilde{\Phi}_i \Pi_i^\theta\td_i^\theta\tilde{\Phi}_i \tilde{Q}_i^\theta - \tilde{\Phi}_i\Pi_i^\theta\td_i^\theta Q_i^\theta + \tilde{\Phi}_i \Pi_i^\theta Q_i^\theta - Q_i^\theta\\
    &= \tilde{\Phi}_i \Pi_i^\theta P^\theta(\tilde{\Phi}_i \tilde{Q}_i^\theta - Q_i^\theta ) + \tilde{\Phi}_i \Pi_i^\theta Q_i^\theta - Q_i^\theta.
\end{align*}
Note that by Assumption~\ref{assump:ergodicity}, \[\Vert \tilde{\Phi}_i \Pi_i^\theta P^\theta(\tilde{\Phi}_i \tilde{Q}_i^\theta - Q_i^\theta )\Vert_{\SD^\theta} = \Vert \tilde{\Phi}_i \Pi_i^\theta (P^\theta - \mathbf{1}(\sd^\theta)^\top)(\tilde{\Phi}_i \tilde{Q}_i^\theta - Q_i^\theta )\Vert_{\SD^\theta} \leq\mu_D \Vert \tilde{\Phi}_i \tilde{Q}_i^\theta - Q_i^\theta \Vert_{\SD^\theta}.\]
This shows $\Vert \tilde{\Phi}_i \tilde{Q}_i^{\theta} - Q_i^\theta \Vert_{\SD^\theta} \leq \mu_D \Vert \tilde{\Phi}_i \tilde{Q}_i^\theta - Q_i^\theta \Vert_{\SD^\theta} + \Vert\tilde{\Phi}_i \Pi_i^\theta Q_i^\theta - Q_i^\theta \Vert_{\SD^\theta}$, and hence, 
\begin{align}
    \Vert \tilde{\Phi}_i \tilde{Q}_i^{\theta} - Q_i^\theta \Vert_{\SD^\theta} \leq \frac{1}{1-\mu_D} \Vert\tilde{\Phi}_i \Pi_i^\theta Q_i^\theta - Q_i^\theta \Vert_{\SD^\theta}\leq \frac{1}{1-\mu_D} \Vert\tilde{\Phi}_i \Pi_i^\theta Q_i^\theta - Q_i^\theta \Vert_{\infty}. \label{eq:critic:fixed_point_error_2}
\end{align}
Next, recall that the $z_{\nik}$'s row of $\Pi_i^\theta$ is the distribution of the state-action pair $z$ conditioned on its $\nik$ coordinates being fixed to be $z_{\nik}$. We denote this conditional distribution of the states outside of $\nik$, $z_{\nminusik}$, given $z_{\nik}$, as $\sd^\theta(z_{\nminusik}|z_{\nik})$. With this notation,
% $$(\Pi_i^\theta Q^*_i)(z_{\nik}) = \sum_{z_{\nminusik}}d(z_{\nminusik}|z_{N_{i}^k} ) Q^*_i(z_{\nik},z_{\nminusik}) .$$

\begin{align*}
    (\tilde{\Phi}_i\Pi_i^\theta Q_i^\theta)(z_{\nik}, z_{\nminusik}) =\sum_{z_{\nminusik}'}\sd^\theta(z_{\nminusik}'|z_{\nik} ) Q_i^\theta(z_{\nik},z_{\nminusik}').
\end{align*}
Therefore, we have,
\begin{align*}
  & |  (\tilde{\Phi}_i\Pi_i^\theta Q_i^\theta)(z_{\nik}, z_{\nminusik}) - Q_i^\theta (z_{\nik}, z_{\nminusik})|\\
  &= \bigg|  \sum_{z_{\nminusik}'}\sd^\theta(z_{\nminusik}'|z_{\nik} ) Q_i^\theta(z_{\nik},z_{\nminusik}') - \sum_{z_{\nminusik}'}\sd^\theta(z_{\nminusik}'|z_{\nik} ) Q_i^\theta(z_{\nik},z_{\nminusik}) \bigg|\\
  &\leq
 \sum_{z_{\nminusik}'}\sd^\theta(z_{\nminusik}'|z_{\nik} ) \big|Q_i^\theta(z_{\nik},z_{\nminusik}') -Q_i^\theta(z_{\nik},z_{\nminusik})\big| \\
  &\leq c \rhok,
\end{align*}
where the last inequality is due to the exponential decay property (cf. Definition~\ref{def:exp_decay} and Assumption~\ref{assump:exp_decay}). Therefore,
\[\Vert \tilde{\Phi}_i \Pi_i^\theta Q_i^\theta - Q_i^\theta\Vert_\infty\leq  c \rhok.\]
Combining the above with \eqref{eq:critic:fixed_point_error_2}, we get,
\[  \Vert \tilde{\Phi}_i \tilde{Q}_i^{\theta} - Q_i^\theta\Vert_{\SD^\theta} \leq \frac{  c \rhok}{1-\mu_{D}},  \]
which, when combined with \eqref{eq:critic:fixed_point_error_1}, leads to the desired result.

\section{Analysis of the Actor and Proof of Theorem~\ref{thm:convergence}} \label{sec:actor}
The proof is divided into three steps. Firstly, we decompose the error in the gradient approximation into three sequences. Then, we bound the three error sequences seperately. Finally, using the bounds, we prove Theorem~\ref{thm:convergence}.

\textbf{Step 1: Error decomposition. }Recall that the actor update can be written as $\theta_i(t+1) = \theta_i(t) + \eta_t \Gamma(\mathbf{\hat{Q}}^t) \hat{g}_i(t)$, where
\begin{align*}
    \hat{g}_i(t) =    \nabla_{\theta_i} \log \zeta_i^{\theta_i(t)}(a_i(t)|s_i(t)) \frac{1}{n} \sum_{j \in \nik}  \hat Q_j^t(s_{\njk}(t),a_{\njk}(t)),
\end{align*}
and $\Gamma(\mathbf{\hat{Q}}^t) = \frac{1}{1 + \max_{j} \Vert\hat{Q}_j^t \Vert_\infty}$ is a scalar whose purpose is to control the size of the approximated gradient. We also denote $\Gamma_t = \Gamma(\mathbf{\hat{Q}}^t)$. Recall that the true gradient of the objective function is given by (Lemma~\ref{lem:policy_gradient}),
\begin{align*}
    \nabla_{\theta_i} J(\theta(t)) = \E_{(s,a)\sim \sd^{\theta(t)} }\nabla_{\theta_i} \log \zeta_i^{\theta_i(t)}(a_i|s_i) \frac{1}{n} \sum_{j\in\mathcal{N}} Q_j^{\theta(t)} (s,a) .
\end{align*}
The error between the approximated gradient $\hat{g}_i(t)$ and the true gradient $\nabla_{\theta_i} J(\theta(t))$ can be decomposed into three terms,
\begin{align*}
   & \hat{g}_i(t) - \nabla_{\theta_i} J(\theta(t)) \\
    &=\nabla_{\theta_i} \log \zeta_i^{\theta_i(t)}(a_i(t)|s_i(t)) \frac{1}{n} \sum_{j \in \nik}  \hat Q_j^t(s_{\njk}(t),a_{\njk}(t)) - \E_{(s,a)\sim \sd^{\theta(t)}}\nabla_{\theta_i} \log \zeta_i^{\theta_i(t)}(a_i|s_i)  \frac{1}{n} \sum_{j \in \nik}  \hat Q_j^t(s_{\njk},a_{\njk}) \\
    &+  \E_{(s,a)\sim \sd^{\theta(t)}}\nabla_{\theta_i} \log \zeta_i^{\theta_i(t)}(a_i|s_i) \frac{1}{n} \Big[ \sum_{j \in \nik}  \hat Q_j^t(s_{\njk},a_{\njk}) - \sum_{j\in\nik} \hat{Q}_j^{\theta(t)}(s_{\njk},a_{\njk})  \Big]\\
    &+  \E_{(s,a)\sim \sd^{\theta(t)}}\nabla_{\theta_i} \log \zeta_i^{\theta_i(t)}(a_i|s_i) \frac{1}{n} \Big[ \sum_{j \in \nik}  \hat Q_j^{\theta(t)}(s_{\njk},a_{\njk}) - \sum_{j\in\mathcal{N}} Q_j^{\theta(t)} (s,a)  \Big]\\
    &:= e_i^1(t) + e_i^2(t) + e_i^3(t).
\end{align*}
We also use $e^1(t)$, $e^2(t)$, $e^3(t)$, $\hat{g}(t)$ to denote $e_i^1(t)$, $e_i^2(t)$, $e_i^3(t)$, $\hat{g}_i(t)$ stacked into a larger vector respectively. We next bound the three error sequences $e_i^1(t)$, $e_i^2(t)$ and $e_i^3(t)$.

\textbf{Step 2: Bounding error sequences.} In this step, we provide bounds on the error sequences. We will frequently use the following auxiliary result, whose proof is omitted as it is identical to that of \Cref{lem:critic:Gh_bounded}.  
\begin{lemma}\label{lem:q_ub}
We have for any $\theta$ and $i$, $\Vert Q_i^\theta\Vert_\infty \leq Q_{\max} = \frac{2c_\infty\bar{r}}{1-\mu_D} $. As a result, $\Vert \nabla_{\theta_i}J(\theta)\Vert \leq L_i Q_{\max}$ and $\Vert \nabla J(\theta)\Vert \leq L Q_{\max}$.
\end{lemma}
We start with a bound related to error sequence $e_i^1(t)$, the proof of which is postponed to \Cref{subsec:actor:e1}. 
\begin{lemma} \label{lem:actor:e1}
Almost surely, for all $i$, we have $\sum_{t=0}^T \eta_t \langle \nabla_{\theta_i}J(\theta(t)), \Gamma_t e_i^1(t)\rangle $ converges to a finite limit as $T\rightarrow\infty$. 
\end{lemma}

Then, we bound error sequence $e_i^2(t)$ in the following Lemma~\ref{lem:actor:e2}, which is an immediate consequence from our analysis of critic in Theorem~\ref{thm:critic} of \Cref{sec:critic}.
\begin{lemma}\label{lem:actor:e2}
Almost surely, $\lim_{t\rightarrow\infty} e_i^2(t) = 0$.
\end{lemma}
\begin{proof} We have 
    $\Vert e_i^2(t)\Vert \leq L_i \frac{1}{n} \sum_{j\in\nik} \Vert \hat{Q}_j^t - \hat{Q}_j^{\theta(t)}\Vert_\infty \rightarrow 0  $ as $t\rightarrow\infty$, where we have used part (b) of Theorem~\ref{thm:critic}. 
\end{proof}
Lastly, in Lemma~\ref{lem:actor:e3} we show that $e_i^3(t)$ can be bounded by a small constant as a result of Theorem~\ref{thm:critic}(a). The proof of Lemma~\ref{lem:actor:e3} is postponed to \Cref{subsec:actor:e3}. 
\begin{lemma}\label{lem:actor:e3}
Almost surely, for each $i$, $\Vert e_i^3(t)\Vert \leq L_i \frac{  c \rhok}{1-\mu_{D}}.$
\end{lemma}

With these preparations, we are now ready to prove Theorem~\ref{thm:convergence}. 

\textbf{Step 3: Proof of Theorem~\ref{thm:convergence}.} Recall that $\beta_t = \eta_t \Gamma_t$. Note that by the definition of $\Gamma(\cdot)$, $\beta_t\leq \eta_t$. Further, almost surely there exists some constant $\gamma$ s.t. $\beta_t \geq \gamma \eta_t$, as by~Theorem~\ref{thm:critic}(b), almost surely, $\Vert \hat{Q}_j^t\Vert_\infty$ is uniformly upper bounded for all $t\geq 0, j\in\mathcal{N}$ by some constant. Since the objective function is $L'$-smooth, we have,
\begin{align*}
    &J(\theta(t+1))
    \\ &\geq J(\theta(t))+ \langle \nabla J(\theta(t), \beta_t \hat{g}(t)\rangle - \frac{L'}{2}\beta_t^2 \Vert \hat{g}(t)\Vert^2\\
    &= J(\theta(t))+\beta_t \Vert \nabla J(\theta(t))\Vert^2 + \beta_t\langle \nabla J(\theta(t)),e^1(t)+e^2(t)+e^3(t) \rangle - \frac{L'}{2}\beta_t^2 \Vert \hat{g}(t)\Vert^2.
    \end{align*}
Therefore, by a telescope sum we have,
\begin{align*}
J(\theta(T+1))   
&\geq J(\theta(0)) + \sum_{t=0}^T \beta_t \Vert \nabla J(\theta(t))\Vert^2  + \sum_{t=0}^T \eta_t\langle \nabla J(\theta(t)),\Gamma_t e^1(t) \rangle+ \sum_{t=0}^T \beta_t\langle \nabla J(\theta(t)),e^2(t) \rangle \\
&\quad +\sum_{t=0}^T \beta_t\langle \nabla J(\theta(t)),e^3(t) \rangle  - \sum_{t=0}^T \frac{L'}{2}\beta_t^2\Vert\hat{g}(t)\Vert^2\\
&\geq \sum_{t=0}^T \beta_t \Vert \nabla J(\theta(t))\Vert^2  + \sum_{t=0}^T \eta_t\langle \nabla J(\theta(t)),\Gamma_t e^1(t) \rangle - \sum_{t=0}^T \beta_t   LQ_{\max}  \Vert e^2(t) \Vert \\
&\quad - \sum_{t=0}^T\beta_t \Vert \nabla J(\theta(t))\Vert \Vert e^3(t) \Vert   - \sum_{t=0}^T \frac{L'}{2}\eta_t^2 L^2,
\end{align*}
where in the last step, we have used $\Vert\nabla J(\theta(t)) \Vert  \leq LQ_{\max}$ (cf. \Cref{lem:q_ub}); we have also used that $\Vert \Gamma_t \hat{g}(t)\Vert \leq L$. Then, rearranging the above inequality, we get,
\begin{align}
   &\frac{ \sum_{t=0}^T \beta_t (\Vert \nabla J(\theta(t))\Vert^2 - \Vert \nabla J(\theta(t))\Vert\Vert e^3(t)\Vert ) }{\sum_{t=0}^T \beta_t} \nonumber \\
   &\leq \frac{ J(\theta(T+1))- \sum_{t=0}^T \eta_t\langle \nabla J(\theta(t)),\Gamma_t e^1(t) \rangle + \sum_{t=0}^T \frac{L'}{2}\eta_t^2 L^2}{\gamma \sum_{t=0}^T \eta_t}   + L Q_{\max}\frac{\sum_{t=0}^T \beta_t  \Vert e^2(t)\Vert}{\sum_{t=0}^T \beta_t}  \nonumber  \\
   &\leq \frac{ \bar{r}- \sum_{t=0}^T \eta_t\langle \nabla J(\theta(t)),\Gamma_t e^1(t) \rangle + \sum_{t=0}^T \frac{L'}{2}\eta_t^2 L^2 }{\gamma \sum_{t=0}^T \eta_t}  + L Q_{\max} \frac{\sum_{t=0}^T \beta_t \Vert e^2(t)\Vert}{\sum_{t=0}^T \beta_t}  ,\label{eq:actor:final_error}
\end{align}
where we have used $J(\theta(T+1)) \leq \bar{r} $ (Assumption~\ref{assump:reward}) and $\beta_t \geq \gamma \eta_t$. 
In \eqref{eq:actor:final_error}, when $T\rightarrow\infty$, the first term on the right hand side goes to zero as its denominator goes to infinity (Assumption \ref{assump:stepsize}) while its nominator is bounded (using Lemma~\ref{lem:actor:e1} and $\sum_{t=0}^\infty \eta_t^2 < \infty$); the second term goes to zero as $\Vert e^2(t)\Vert \rightarrow 0 $ (Lemma~\ref{lem:actor:e2}) and $\sum_{t=0}^T \beta_t \geq \gamma \sum_{t=0}^T \eta_t \rightarrow\infty$. So the right hand side of \eqref{eq:actor:final_error} converges to $0$. 
%The the third term is the steady state error. 
From this, we have by Lemma~\ref{lem:actor:e3},
\[\liminf_{t\rightarrow\infty} \Vert J(\theta(t))\Vert  \leq \sup_{t\geq 0} \Vert e^3(t)\Vert \leq    L \frac{  c \rhok}{1-\mu_{D}}, \]
because otherwise, the left hand side of \eqref{eq:actor:final_error} will be positive and bounded away from zero as $T\rightarrow\infty$, a contradction.

\subsection{Proof of Lemma~\ref{lem:actor:e1}}\label{subsec:actor:e1}
    We fix $i$ and define for $z=(s,a)$, $\mathbf{\hat{Q}} = \{\hat{Q}_i\}_{i=1}^n$, 
\[F^\theta(\mathbf{\hat{Q}},z) =\langle \nabla_{\theta_i} J(\theta), \Gamma(\mathbf{\hat{Q}}) \nabla_{\theta_i} \log \zeta_i^{\theta_i}(a_i|s_i) \frac{1}{n} \sum_{j \in \nik}  \hat Q_j(s_{\njk},a_{\njk})\rangle, \]
and $\bar{F}^\theta(\mathbf{\hat{Q}}) = \E_{z\sim\pi^\theta} F^\theta(\mathbf{\hat{Q}},z)$. We also define $\hat{F}^\theta(\mathbf{\hat{Q}},\cdot)$ to be the solution of the Poission equation:
    \[\hat{F}^\theta(\mathbf{\hat{Q}},z) = F^\theta(\mathbf{\hat{Q}},z)-\bar{F}^\theta(\mathbf{\hat{Q}}) + P^\theta \hat{F}^\theta(\mathbf{\hat{Q}},z) =\E_\theta \Big[\sum_{t=0}^\infty (F^\theta(\mathbf{\hat{Q}},z(t))-\bar{F}^\theta(\mathbf{\hat{Q}})) \Big| z(0) = z\Big],  \]
where $P^\theta$ is the transition kernal on the state-action pair under policy $\theta$, and $P^\theta \hat{F}^\theta(\mathbf{\hat{Q}},z) = \E_{z'\sim P^\theta(\cdot|z)} \hat{F}^\theta(\mathbf{\hat{Q}},z')$. 

One can easily check that $\hat{F}^\theta(\cdot,\cdot)$ satisfies the following properties, the proof of which is deferred to the end of this subsection.
         \begin{lemma}\label{lem:actor:hatF}
     There exists $C_F,L_{\theta,F}, L_{Q,F}>0$ s.t. for all $\theta,z$, $| \hat{F}^\theta(\mathbf{\hat{Q}},z) | \leq C_F$ and $\hat{F}^\theta(\mathbf{\hat{Q}},z)$ is $L_{\theta,F}$-Lipschitz continuous in $\theta$ in Euclidean norm, and $L_{Q,F}$-Lipschitz continuous in $\mathbf{\hat{Q}}$ in the sense that,  
     \[|\hat{F}^\theta(\mathbf{\hat{Q}},z) - \hat{F}^\theta(\mathbf{\hat{Q}}',z)|\leq L_{Q,F}\sum_{j=1}^n \Vert \hat{Q}_j - \hat{Q}_j'\Vert_\infty .\]
     \end{lemma}
    
With this definition, we can decompose $\langle \nabla_{\theta_i}J(\theta(t)), \Gamma_t e_i^1(t)\rangle $ into the following terms, 
    \begin{align*}
        \langle \nabla_{\theta_i}J(\theta(t)), \Gamma_t e_i^1(t)\rangle &= F^{\theta(t)}(\mathbf{\hat{Q}}^t,z(t)) - \bar{F}^{\theta(t)}(\mathbf{\hat{Q}}^t)\\
        &= \hat{F}^{\theta(t)}(\mathbf{\hat{Q}}^t,z(t)) - P^{\theta(t)}\hat{F}^{\theta(t)}(\mathbf{\hat{Q}}^t,z(t))\\
        &= \hat{F}^{\theta(t)}(\mathbf{\hat{Q}}^t,z(t+1)) - P^{\theta(t)}\hat{F}^{\theta(t)}(\mathbf{\hat{Q}}^t,z(t))\\
        & + \hat{F}^{\theta(t-1)}(\mathbf{\hat{Q}}^{t-1},z(t)) - \hat{F}^{\theta(t)}(\mathbf{\hat{Q}}^t,z(t+1))\\
        & + \hat{F}^{\theta(t)}(\mathbf{\hat{Q}}^t, z(t)) - \hat{F}^{\theta(t-1)}(\mathbf{\hat{Q}}^{t-1},z(t))\\
        &= a^1(t) + a^2(t) + a^3(t) + a^4(t),
    \end{align*}
    where we have defined, 
    \begin{align*}
        a^1(t) &= \hat{F}^{\theta(t)}(\mathbf{\hat{Q}}^t,z(t+1)) - P^{\theta(t)}\hat{F}^{\theta(t)}(\mathbf{\hat{Q}}^t,z(t)),\\
        a^2(t) &= \frac{1}{\eta_t}(\eta_{t-1}\hat{F}^{\theta(t-1)}(\mathbf{\hat{Q}}^{t-1},z(t)) -\eta_t \hat{F}^{\theta(t)}(\mathbf{\hat{Q}}^t,z(t+1))),\\
        a^3(t) & = \frac{\eta_t - \eta_{t-1}}{\eta_t}\hat{F}^{\theta(t-1)}(\mathbf{\hat{Q}}^{t-1},z(t)),\\
        a^4(t) &= \hat{F}^{\theta(t)}(\mathbf{\hat{Q}}^t, z(t)) - \hat{F}^{\theta(t-1)}(\mathbf{\hat{Q}}^{t-1},z(t)).
    \end{align*}
    
    With the decomposition, in what follows we show that $\sum_{t=1}^{T} \eta_t a^j(t)$ converges to a finite limit almost surely for $j=1,2,3,4$, which together will conclude the proof of this lemma.
     
    For $a^1(t)$, let $\mathcal{F}_t$ be the $\sigma$-algebra generated by $\{\theta(k),\mathbf{\hat{Q}}^k, z(k)\}_{k\leq t}$. Then, $a^1(t)$ is $\mathcal{F}_{t+1}$-measurable and $\E a^1(t)|\mathcal{F}_t=0 $. As such, $\sum_{t=1}^{T} \eta_t a^1(t)$ is a martingale process, and further,
     \begin{align*}
       \E  | \sum_{t=1}^{T} \eta_t a^1(t) |^2 =   \sum_{t=1}^T \eta_t^2 \E | a^1(t)|^2 \leq 4 C_F^2 \sum_{t=0}^\infty \eta_t^2<\infty.
     \end{align*}
     As such, by martingale convergence theorem, $ \sum_{t=1}^{T} \eta_t a^1(t)$ converges to a finite limit as $T\rightarrow \infty$ almost surely.
     
     For $a^2(t)$, note that \[\sum_{t=1}^{T} \eta_t a^2(t)= \eta_0\hat{F}^{\theta(0)}(\mathbf{\hat{Q}}^0,z(1)) - \eta_{T}\hat{F}^{\theta(T)}(\mathbf{\hat{Q}}^{T},z(T+1)),\] which also converges to a finite limit as $T\rightarrow\infty$, almost surely. 
     
     For $a^3(t)$, since the step size $\eta_t$ is non-increasing, we have, 
     \begin{align*}
          \sum_{t=1}^{T}\eta_t |a^3(t)| = \sum_{t=1}^{T} ( \eta_{t-1}-\eta_t ) |\hat{F}^{\theta(t-1)}(\mathbf{\hat{Q}}^{t-1},z(t)) |
         \leq C_F (\eta_0 - \eta_{T}) < C_F \eta_0.
         %= \sum_{t=0}^{T-1} \beta_t (\hat{F}^{\theta(t-1)}(\mathbf{\hat{Q}}^{t-1},z(t)) - \hat{F}^{\theta(t)}(\hat{Q}^{t},z(t+1))) + \beta_{T}\hat{F}^{\theta(T-1)}(\hat{Q}^{T-1},z(T)) - \beta_0\hat{F}^{\theta(0)}(\hat{Q}^0,z(1))
     \end{align*}
     As such  $\sum_{t=1}^{T}\eta_t a^3(t)$ converges to a finite limit almost surely. 
     
     Finally, for $a^4(t)$, we note that by the Lipschitz property of $\hat{F}^\theta(\mathbf{\hat{Q}},z)$ in \Cref{lem:actor:hatF},
     \begin{align*}
         |a^4(t)| &\leq |\hat{F}^{\theta(t)}(\mathbf{\hat{Q}}^t, z(t)) -\hat{F}^{\theta(t-1)}(\mathbf{\hat{Q}}^t, z(t))| + | \hat{F}^{\theta(t-1)}(\mathbf{\hat{Q}}^t, z(t)) - \hat{F}^{\theta(t-1)}(\mathbf{\hat{Q}}^{t-1},z(t))|\\
         &\leq L_{\theta,F}\Vert \theta(t) - \theta(t-1)\Vert + L_{Q,F} \sum_{j=1}^n \Vert \hat{Q}_j^t - \hat{Q}_j^{t-1}\Vert_\infty \\
         &\leq \bar{L}_{\theta,F} \eta_{t-1} + \bar{L}_{Q,F}\alpha_{t-1},
     \end{align*}
     for some constant $\bar{L}_{\theta,F}$ and $\bar{L}_{Q,F}$ almost surely. Here we have used $\Vert \theta(t) - \theta(t-1)\Vert \leq \eta_{t-1}L$ (check how we verified Assumption~\ref{assump:sa4} in \Cref{subsec:proof_critic}). Further, we have used $\Vert \hat{Q}_j^t - \hat{Q}_j^{t-1}\Vert_\infty \leq \alpha_{t-1}(2\bar{r} + 2 \Vert \hat{Q}_j^{t-1}\Vert_\infty)$ (cf. equation \eqref{eq:critic_update_2} in \Cref{subsec:proof_critic}), and the fact that $\Vert \hat{Q}_j^{t-1}\Vert_\infty$ is upper bounded uniformly over $t$ almost surely, cf. Theorem~\ref{thm:critic}.  As such, we have 
     \[\sum_{t=1}^{T} \eta_t |a^4(t)| \leq \sum_{t=1}^\infty \Big(\bar{L}_{\theta,F} \eta_t \eta_{t-1} + \bar{L}_{Q,F} \alpha_{t-1}\eta_t\Big) < \infty. \]
     As a result, we have, $\sum_{t=1}^{T} \eta_t a^4(t)$ converges to a finite limit as $T\rightarrow\infty$, almost surely.  This concludes the proof of \Cref{lem:actor:e1}. 
     
     Finally, we provide the proof for Lemma~\ref{lem:actor:hatF}.
     
\textit{Proof of Lemma~~\ref{lem:actor:hatF}.}
    Clearly, $| F^\theta (\mathbf{\hat{Q}},z)|\leq \Vert\nabla_{\theta_i}J(\theta)\Vert L_i  \leq L_i^2  Q_{\max}:= C_F' $, where we have used $\Vert\nabla_{\theta_i} J(\theta)\Vert\leq L_i Q_{\max}$ (cf. \Cref{lem:q_ub}). Using the same argument as in Lemma~\ref{lem:critic:Gh_bounded} (b), we have $|\hat{F}^\theta (\mathbf{\hat{Q}},z)|\leq C_F = \frac{2c_\infty}{1-\mu_D} C_F' $. Next, note that,
    \begin{align*}
        &|F^\theta (\mathbf{\hat{Q}},z) - F^{\bar{\theta}} (\mathbf{\hat{Q}},z)| \\
        &\leq | \langle \nabla_{\theta_i} J(\theta) - \nabla_{\theta_i} J(\bar{\theta}), \Gamma(\mathbf{\hat{Q}}) \nabla_{\theta_i} \log \zeta_i^{\theta_i}(a_i|s_i) \frac{1}{n} \sum_{j \in \nik}  \hat Q_j(s_{\njk},a_{\njk})\rangle |\\
        &\quad +  | \langle  \nabla_{\theta_i} J(\bar{\theta}), \Gamma(\mathbf{\hat{Q}}) (\nabla_{\theta_i} \log \zeta_i^{\theta_i}(a_i|s_i)- \nabla_{\theta_i} \log\zeta_i^{\bar{\theta}_i}(a_i|s_i)  ) \frac{1}{n} \sum_{j \in \nik}  \hat Q_j(s_{\njk},a_{\njk})\rangle |\\
        &\leq L' L_i  \Vert \theta - \bar\theta\Vert + L_iQ_{\max} L_i' \Vert \theta - \bar\theta\Vert:= L_{\theta,F}' \Vert\theta - \bar\theta\Vert .
    \end{align*}
    The above shows $F^\theta (\mathbf{\hat{Q}},z)$ is Lipschitz in $\theta$. Then, using a similar argument as Lemma~\ref{lem:critic:lipschitz} (c), we can show $\hat{F}^\theta (\mathbf{\hat{Q}},z)$ is Lipschitz continuous in $\theta$. To do this, we fix any initial $z$, let $d^{\theta,t}$ be the distribution of $z(t)$ under policy $\theta$. 
    Then, equation \eqref{eq:critic:d_lipschitz} in the proof of \Cref{lem:critic:lipschitz} shows that $d^{\theta,t} - \pi^\theta$ is Lipschitz in $\theta$ with Lipschitz constant geometrically decaying in $t$, i.e. for some $L_d>0$,
    \[ \Vert (d^{\theta,t} - \pi^\theta) - (d^{\bar\theta,t} - \pi^{\bar\theta}) \Vert_1 \leq L_d (\frac{\mu_D+1}{2})^t \Vert \theta - \bar\theta\Vert. \]
Note that by definition, 
\begin{align}
     \hat{F}^\theta (\mathbf{\hat{Q}},z)  &= \E_\theta \Big[\sum_{t=0}^\infty (F^\theta(\mathbf{\hat{Q}},z(t))-\bar{F}^\theta(\mathbf{\hat{Q}})) \Big| z(0) = z\Big]\nonumber \\
     &= \sum_{t=0}^\infty [ \E_{z'\sim d^{\theta,t}}F^\theta(\mathbf{\hat{Q}},z')  -\E_{z'\sim \pi^\theta}F^\theta(\mathbf{\hat{Q}},z')] \nonumber \\
     &=\sum_{t=0}^\infty \sum_{z'\in\mathcal{Z}} (d^{\theta,t}(z') - \pi^\theta(z')) F^\theta(\mathbf{\hat{Q}},z'). \label{eq:actor:hatF_dpi}
\end{align}
As such,
\begin{align*}
    | \hat{F}^\theta (\mathbf{\hat{Q}},z) - \hat{F}^{\bar\theta} (\mathbf{\hat{Q}},z) | 
    &\leq \sum_{t=0}^\infty \sum_{z'\in\mathcal{Z}} \Big| (d^{\theta,t}(z') - \pi^\theta(z')) F^\theta(\mathbf{\hat{Q}},z') - (d^{\bar\theta,t}(z') - \pi^{\bar\theta}(z')) F^{\bar\theta}(\mathbf{\hat{Q}},z') \Big|\\
    &\leq  \sum_{t=0}^\infty \sum_{z'\in\mathcal{Z}} \Big[ |d^{\theta,t}(z') - \pi^\theta(z') - (d^{\bar\theta,t}(z') - \pi^{\bar\theta}(z')) | | F^\theta(\mathbf{\hat{Q}},z') | \\
    &\qquad + |d^{\bar\theta,t}(z') - \pi^{\bar\theta}(z')| | F^\theta(\mathbf{\hat{Q}},z')- F^{\bar\theta}(\mathbf{\hat{Q}},z')| \Big]\\
    &\leq \sum_{t=0}^\infty \Big[\Vert d^{\theta,t} - \pi^\theta - (d^{\bar\theta,t} - \pi^{\bar\theta})\Vert_1 C_F' + \Vert d^{\bar\theta,t} - \pi^{\bar\theta}\Vert_1  L_{\theta,F}' \Vert\theta - \bar\theta\Vert \Big]\\
    &\leq \sum_{t=0}^\infty \Big[  L_d C_F' (\frac{1+\mu_D}{2})^t \Vert \theta - \bar{\theta}\Vert  + 2 c_\infty \mu_D^t L_{\theta,F}' \Vert\theta - \bar\theta\Vert \Big] \\
    &\leq L_{\theta,F}\Vert\theta - \bar{\theta}\Vert,
\end{align*}
for some $L_{\theta,F}>0$. This shows that $\hat{F}^\theta (\mathbf{\hat{Q}},z)$ is Lipschitz continuous in $\theta$. 
    
Finally, we show $\hat{F}^\theta (\mathbf{\hat{Q}},z)$ is Lipschitz in $\mathbf{\hat{Q}}$. Note that,
    \begin{align*}
        |F^\theta (\mathbf{\hat{Q}},z) - F^{\theta} (\mathbf{\hat{Q}}',z)|
        &\leq L_i^2 Q_{\max} \Big|\Gamma(\mathbf{\hat{Q}})    \frac{1}{n} \sum_{j \in \nik}  \hat Q_j(s_{\njk},a_{\njk}) - \Gamma(\mathbf{\hat{Q}}')   \frac{1}{n} \sum_{j \in \nik}  \hat Q_j'(s_{\njk},a_{\njk}) \Big|\\
        &\leq L_i^2Q_{\max} \Big|(\Gamma(\mathbf{\hat{Q}})- \Gamma(\mathbf{\hat{Q}}'))    \frac{1}{n} \sum_{j \in \nik}  \hat Q_j(s_{\njk},a_{\njk})\Big|\\
        &\qquad + L_i^2Q_{\max} \Gamma(\mathbf{\hat{Q}}')  \Big| \frac{1}{n} \sum_{j \in \nik} (\hat Q_j(s_{\njk},a_{\njk})- \hat Q_j'(s_{\njk},a_{\njk})) \Big|.
        %&\leq L_i^2Q_{\max} m \sum_{j=1}^n \Vert \hat{Q}_j - \hat{Q}_j'\Vert_\infty + L_i^2 Q_{\max}m \sum_{j=1}^n \Vert Q_j - Q_j'\Vert_\infty. 
    \end{align*}
Note that
\begin{align*}
    |\Gamma(\mathbf{\hat{Q}})- \Gamma(\mathbf{\hat{Q}}')| = \frac{ |\max_j \Vert \hat{Q}_j\Vert_\infty - \max_j \Vert \hat{Q}_j'\Vert_\infty |}{(1 + \max_j \Vert \hat{Q}_j\Vert_\infty) (1 + \max_j \Vert \hat{Q}_j'\Vert_\infty) } \leq \frac{ \sum_j \Vert \hat{Q}_j - \hat{Q}_j'\Vert_\infty }{(1 + \max_j \Vert \hat{Q}_j\Vert_\infty)}.
\end{align*}
As such,
\begin{align*}
    |F^\theta (\mathbf{\hat{Q}},z) - F^{\theta} (\mathbf{\hat{Q}}',z)| &\leq L_i^2 Q_{\max} \sum_j \Vert \hat{Q}_j - \hat{Q}_j'\Vert_\infty + L_i^2 Q_{\max}\frac{1}{n} \sum_{j} \Vert\hat{Q}_j - \hat{Q}_j'\Vert_\infty\\
    &:= L_{Q,F}'  \sum_{j} \Vert\hat{Q}_j - \hat{Q}_j'\Vert_\infty,
\end{align*}
which shows $F^\theta (\mathbf{\hat{Q}},z) $ is Lipschitz in $\mathbf{\hat{Q}}$. Then, by \eqref{eq:actor:hatF_dpi},
\begin{align*}
    | \hat{F}^\theta (\mathbf{\hat{Q}},z) - \hat{F}^{\theta} (\mathbf{\hat{Q}'},z) | 
    &\leq \sum_{t=0}^\infty \sum_{z'\in\mathcal{Z}} |d^{\theta,t}(z') - \pi^\theta(z')| \Big| F^\theta(\mathbf{\hat{Q}},z') - F^{\theta}(\mathbf{\hat{Q}'},z') \Big|\\
    &\leq \sum_{t=0}^\infty \Vert d^{\theta,t} - \pi^\theta \Vert_1 L_{Q,F}'  \sum_{j} \Vert\hat{Q}_j - \hat{Q}_j'\Vert_\infty \\
    &\leq \sum_{t=0}^\infty  2 c_\infty \mu_D^t L_{Q,F}'  \sum_{j} \Vert\hat{Q}_j - \hat{Q}_j'\Vert_\infty \\
    &\leq L_{Q,F}\sum_{j} \Vert\hat{Q}_j - \hat{Q}_j'\Vert_\infty ,
\end{align*}
which shows $\hat{F}^\theta (\mathbf{\hat{Q}},z)$ is Lipschitz in $\mathbf{\hat{Q}}$.
\qed
     
\subsection{Proof of Lemma~\ref{lem:actor:e3}}\label{subsec:actor:e3}
    Let $c_j^{\theta(t)}$ be the constant in Theorem~\ref{thm:critic}(a). Then, % and let $\mathbf{1}$ be the constant function over $(s,a)\in\mathcal \mathcal{Z}$. Then, we have,
\begin{align*}
    e_i^3(t)&=  \E_{(s,a)\sim \sd^{\theta(t)}}\nabla_{\theta_i} \log \zeta_i^{\theta_i(t)}(a_i|s_i) \frac{1}{n} \Big[ \sum_{j \in \nik}  \hat Q_j^{\theta(t)}(s_{\njk},a_{\njk}) -  \sum_{j=1}^n Q_j^{\theta(t)} (s,a)  \Big]\\
    &= \E_{(s,a)\sim \sd^{\theta(t)}}\nabla_{\theta_i} \log \zeta_i^{\theta_i(t)}(a_i|s_i) \frac{1}{n} \Big[ \sum_{j =1}^n  \hat Q_j^{\theta(t)}(s_{\njk},a_{\njk}) -  \sum_{j=1}^n Q_j^{\theta(t)} (s,a)  \Big]\\
    &= \frac{1}{n}  \sum_{j =1}^n  \E_{(s,a)\sim \sd^{\theta(t)}}\nabla_{\theta_i} \log \zeta_i^{\theta_i(t)}(a_i|s_i) \Big[ \hat Q_j^{\theta(t)}(s_{\njk},a_{\njk}) + c_j^{\theta(t)}-  Q_j^{\theta(t)} (s,a)  \Big],
\end{align*}
where in the second equality, we have used for all $j\not\in \nik$, \[\mathbb{E}_{(s,a)\sim\pi^{\theta(t)}}  \nabla_{\theta_i}  \log \zeta_i^{\theta_i(t)}(a_i|s_i)    \hat{Q}_j^{\theta(t)}(s_{\njk},a_{\njk}) = 0, \]
and in the third equality, we have used for all $j$, $\mathbb{E}_{(s,a)\sim\pi^{\theta(t)}}  \nabla_{\theta_i}  \log \zeta_i^{\theta_i(t)}(a_i|s_i)  c_j^{\theta(t)} = 0$. The reason of these is due to $\hat{Q}_j^{\theta(t)}(s_{\njk},a_{\njk})$ does not depend on $a_i$ when $j\notin \nik$, and $c_j^{\theta(t)}$ does not depend on $a_i$ for all $j$. For more details, see \eqref{eq:truncated_pg_bias} in the proof of Lemma~\ref{lem:truncated_pg} in \Cref{sec:truncated_pg}.

As such, by Cauchy Schwarz inequality and Theorem~\ref{thm:critic}(a),
\begin{align*}
    \Vert e_i^3(t)\Vert 
    &\leq \frac{1}{n}  \sum_{j =1}^n \sqrt{\E_{(s,a)\sim \sd^{\theta(t)}} \Vert \nabla_{\theta_i}  \log \zeta_i^{\theta_i(t)}(a_i|s_i)\Vert^2 } \sqrt{\E_{(s,a)\sim \sd^{\theta(t)}}   \Big[ \hat Q_j^{\theta(t)}(s_{\njk},a_{\njk}) +  c_j^{\theta(t)}-  Q_j^{\theta(t)} (s,a)  \Big]^2}.\\
    &\leq  L_i \frac{  c \rhok}{1-\mu_{D}}.
\end{align*}

% \begin{align*}
%     \Vert e_i^3(t)\Vert 
%     &\leq L_i \frac{1}{n}\sum_{j=1}^n \Vert \Phi_j\hat{Q}_j^{\theta(t)} + c_j\mathbf{1} - Q_j^{\theta(t)}  \Vert_\infty \leq L_i\frac{ c_{tr} c \rhok}{1-\mu_{tr}}.
% \end{align*}
%So we are done. 

\section{Detailed Numerical Experiments} \label{sec:numerical_detail}

We consider a wireless network with multiple access points \citep{zocca2019temporal}, where there is a set of users  $U = \{u_1, u_2, \cdots, u_n\},$ and a set of network access points $Y = \{y_1, y_2, \cdots, y_m\}$. Each user $u_i$ only has access to a subset $Y_i \subseteq Y$ of the access points. We identify the set of users $U$ with the set of agents $\mathcal{N}$ in our model, and we define the interaction graph as the conflict graph, in which two users $u_i$ and $u_j$ are neighbors if and only if they share an access point. In other words, the neighbors of $u_i$ includes $N_i = \{u_j\in U: Y_i\cap Y_j\neq \emptyset \}$. 

Each user $u_i$ maintains a queue of packets defined as follows. At each time step, $u_i$ receives a new packet with probability $p_i$ and the new packet has a initial deadline $d_i$. At each time step, if the packet is successfully sent out (we will define ``send out'' later), it will be removed from the queue; otherwise, its deadline will decrease by $1$ and is discarded immediately from the queue if its remaining deadline is zero. 
The local state $s_i$ of $u_i$ is a characterization of the queue of the packets, and is represented by a $d_i$ binary tuple $s_i = (e_1, e_2, \cdots, e_{d_i})\in\mathcal{S}_i = \{0,1\}^{d_i}$, where for each $m \in \{1,\ldots, d_i\}$, $e_m \in \{0, 1\}$ indicates whether $u_i$ has a packet with remaining deadline $m$. 

At each time step, user $u_i$ can choose to send the earliest packet in its queue to one of the access points in its available set $ Y_i$, or not sending at all. In other words, the action space is $\mathcal{A}_i = \{\textrm{null}\} \cup Y_i  $, where $\textrm{null}$ represents the action of not sending. When the user has an empty queue, then all actions will be understood as the $\textrm{null}$ action. At each time step, if $u_i$'s queue is non-empty and it takes action $a_i = y_k \in Y_i$, i.e. sending the packet to access point $y_k$, then the packet is transmitted with success probability $q_k$ that depends on the access point $y_k$, \emph{conditioned on} no other users select this same access point; however, if another user chooses to send a packet to the same access point (i.e. a collision), neither packet is sent. A user receives a local reward of $1$ once successfully sending a packet. In this setting, the average reward can be interpreted as the throughput in stationarity. 

In the experiments, we set the deadline as $d_i = 2$, and all parameters $p_i$ (packet arrival probability for user $u_i$) and $q_k$ (success transmission probability for access point $y_k$) is generated uniformly random from $[0,1]$. We consider a grid of $5\times 5$ users  in Figure~\ref{fig:communication_grid_appendix}, where each user has access points on the corners of its area. We run the Scalable Actor Critic algorithm with $\khop=0,1$ to learn a localized soft-max policy, starting from a initial policy where the action is chosen uniformly random. 
We compare the proposed method with a benchmark based on the localized ALOHA protocol \citep{aloha}, where each user has a certain probability of sending the earliest packet and otherwise not sending at all. 
When it sends, it sends the packet to a random access point in its available set, with probability proportion to the success transmission probability of this access point and inverse proportion to the number of users that share this access point.
The results are shown in Figure~\ref{fig:comm_ave_reward_5by5_appendix}. It shows that the proposed algorithm can outperform the ALOHA based benchmark, despite the proposed algorithm does not have access to the transmission probability $q_k$ which the benchmark has access to.

%Due to space constraints, we leave the details of the experiment to \fvtest{\Cref{sec:numerical_detail}}{Appendix E} in the supplementary material. 
%\lina{since one main result is the exponential decaying property. Note that theorem 2 also talks about exponential property. I do think we should have a figure demonstrating this result, even if this simulation is done using a toy example and a line.  }

\begin{figure}
  \begin{minipage}[b]{0.5\textwidth}
    \centering
    \includegraphics[width = .8\textwidth]{grid5by5.png}
    \caption{Setup of users and access points.}
    \label{fig:communication_grid_appendix}
  \end{minipage}~
  \begin{minipage}[b]{0.5\textwidth}
    \centering
    \includegraphics[width = \textwidth]{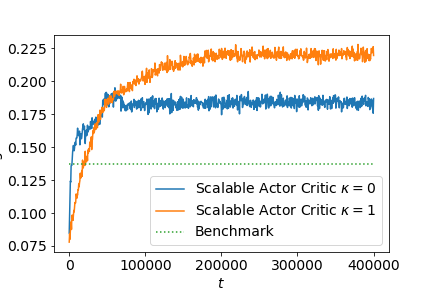}
    \caption{Average reward over the training process. }
    \label{fig:comm_ave_reward_5by5_appendix}
  \end{minipage}
\end{figure}
}{}

\end{document}